\documentclass[a4paper, 11pt]{amsart}
\usepackage{amsmath,amssymb,amscd}
\usepackage{mathtools}
\usepackage[colorlinks=true, allcolors=blue]{hyperref}
\usepackage{url}
\usepackage[utf8]{inputenc}
\usepackage{graphicx}
\usepackage{tikz-cd}
\usepackage{etoc}
\usepackage[
  a4paper,
  left=2.5cm,
  right=2.5cm,
  top=3cm,
  bottom=3cm
]{geometry}
\usepackage{enumerate}
\makeatletter
\newcommand{\pa}{P_{\text{analytic}}}
\newcommand{\pn}{P_{\text{naive}}}
\newcommand{\MapS}{\operatorname{Map}(S)}

\makeatother
\newtheorem{theorem}{Theorem}[section]

\newtheorem{proposition}[theorem]{Proposition}
\newtheorem{corollary}[theorem]{Corollary}
\newtheorem{lemma}[theorem]{Lemma}
\theoremstyle{definition}
\newtheorem{definition}[theorem]{Definition}

\newtheorem{remark}[theorem]{Remark}

\usepackage{xcolor}

\title{Simplicity of action-based \texorpdfstring{$C^{*}$}{C*}-algebras from hyperbolic actions}
\author{Tianyi Lou}


\begin{document}

\begin{abstract}
We study the simplicity of $C^{*}$-algebras built from group actions. For a faithful isometric action of a group $G$ on a countable metric space $X$, we use the associated action representation on $\ell^2(X)$ to define the action-based $C^{*}$-algebra $C^{*}_{X}G$. We define generalized versions of the properties $P_{\mathrm{naive}}$ and $P_{\mathrm{analytic}}$ relative to this action and show that the naive form implies the analytic form. We also prove that $P_{\mathrm{analytic}}$ associated with a coefficient-separating action representation ensures the simplicity of the action-based $C^*$-algebra.
As an application, for every connected orientable boundaryless infinite type surface containing a connected nondisplaceable finite type subsurface, we deduce that its mapping class group satisfies the properties $P_{\mathrm{naive}}^{\mathbb{X}}$ and $P_{\mathrm{analytic}}^{\mathbb{X}}$ for one fixed hyperbolic space $\mathbb X$.
\end{abstract}

\maketitle

\bigskip
{\small
\etocsettocstyle{\noindent\text{Contents}\par\medskip}{\par}
\localtableofcontents
}
\bigskip

\section{Introduction}

Let \(G\) be a countable discrete group. Denote by \(\ell^{2}(G)\) the Hilbert space of square-summable complex-valued functions on \(G\) and by \(\mathcal{B}\left(\ell^{2}(G)\right)\) the algebra of bounded operators on \(\ell^{2}(G)\). The group \(G\) acts on \(\ell^{2}(G)\) via the left regular representation
$$
\lambda(g) f(h)=f\left(g^{-1} h\right), \quad g, h \in G, f \in \ell^{2}(G).
$$
The reduced \(C^{*}\)-algebra \(C_{\mathrm{r}}^{*}(G)\) is the operator norm closure of the linear span of the set of operators \(\{\lambda(g) \mid g \in G\}\) in \(\mathcal{B}\left(\ell^{2}(G)\right)\). 

We say that \(G\) is \(C^{*}\)-simple if its reduced \(C^{*}\)-algebra \(C_{r}^{*}(G)\) is simple, that is, \(C_{r}^{*}(G)\) has no non-trivial two-sided ideals. A landmark result of Powers states that nonabelian free groups are \(C^{*}\)-simple \cite{PR75}. Later, many other examples of \(C^{*}\)-simple groups were found. These include free products of two nontrivial groups, provided that the two factors are not both of order two \cite{WN79}, torsion-free non-elementary Gromov hyperbolic groups \cite{dlH85} (more generally, torsion-free non-elementary convergence groups \cite{dlH07}), centerless mapping class groups and \(\operatorname{Out}(F_n)\) for \(n\geq3\) \cite{BdlH04}, irreducible Coxeter groups which are neither finite nor affine \cite{F03}, relatively hyperbolic groups with no nontrivial finite normal subgroup \cite{AM07a}, and certain nontrivial groups without non-cyclic free subgroups \cite{OO14}, among others.  

In particular, Bekka-Cowling-de la Harpe introduced the property $\pn$ and the property $\pa$ for discrete groups \cite{BCD95}. The property $\pn$ implies the property $\pa$, and the latter
guarantees that \(C_{\mathrm{r}}^{*}(G)\) is simple.


For every non-exceptional finite type surface \(\Sigma\) whose mapping class group \(\operatorname{Map}(\Sigma)\) has no nontrivial finite normal subgroup, \(\operatorname{Map}(\Sigma)\) is discrete and satisfies property $\pn$ by \cite[Theorem~2.3]{AD18}. Consequently, the reduced \(C^{*}\)-algebra of \(\operatorname{Map}(\Sigma)\) is simple. For an infinite type surface \(S\), $\MapS$ is not discrete. In \cite{TL24}, Lou proved that \(\MapS\) satisfies the property \(\pn\) when the infinite type surface contains a nondisplaceable subsurface of finite type. It is therefore natrual to ask the following questions: does the mapping class group \(\MapS\) of an infinite type surface $S$ that contains a nondisplaceable subsurface have the property \(\pa\)? How to define the reduced $C^{*}$-algebra of \(\MapS\)? Is the reduced $C^{*}$-algebra of \(\MapS\) simple? 

Since $S$ is of infinite type, the group $\MapS$ is a non-discrete Polish group. In particular, $\MapS$ is not locally compact. Thus the usual construction of the reduced group $C^*$-algebra for locally compact groups, which relies on Haar measure and the left regular representation on $L^2(G,\mu)$, does not apply directly. To overcome this, we pass from the isometric action of a group $G$ on a countable metric space \(X\) endowed with the counting measure. This group $G$ preserving the counting measure generates a unitary representation in the Hilbert space $\ell^{2}(X)$ of square-summable complex-valued functions on $X$ as follows: 
\[
\pi:G\longrightarrow \mathcal{U}(\ell^{2}(X)),\qquad (\pi(g)f)(x)=f(g^{-1} \cdot x).
\]
We call this map \(\pi\) an \emph{action representation}. 

We define the generalized property \(\pa^{X}\) (see Definition \ref{ana1}) relative to the associated action representation $\pi$ and the generalized property \(\pn^{X}\) (see Definition \ref{naive1}). The following proposition shows that the latter implies the former.

\begin{proposition}[See Proposition \ref{P1}]
  Let $G$ be a group acting faithfully on a countable hyperbolic metric space $X$ by isometries and assume $G$ has the property $\pn^{X}$. Then $G$ satisfies the property \(\pa^X\).\end{proposition}

The action representation of a group \(G\) can be extended from the group to the group algebra by $\pi(\sum_{g\in G}a_g g)=\sum_{g\in G}a_g\pi(g)$. The action of \(G\) on \(X\) is faithful if and only if the associated action representation \(\pi\) is faithful (see Lemma \ref{f}). The group action induces a bounded representation \(\pi\) of \(\mathbb C G\) on \(\ell^2(X)\) (see Lemma \ref{b}), so the action-based \(C^*\)-algebra \(C_X^*G\) is defined to be the norm closure of the image of \(\mathbb C G\) in \(\mathcal B(\ell^2(X))\).

\vspace{2mm}

For a general action representation, faithfulness does not automatically imply that the coefficient maps separate points. We therefore require the action representation to be coefficient-separating (see Definition \ref{coefficient-separating}) in the simplicity criterion below.

\begin{theorem}[See Theorem \ref{thm:Pana-implies-simple-finite}]
Let \(G\) be a group acting faithfully on a countable metric space \(X\), and let \(\pi:G\to \mathcal U(\ell^2(X))\) be the associated action representation. Assume that \(G\) has property \(\pa^X\). If the action representation is coefficient-separating, then \(C_X^*G\) is simple.

In particular, if \(G\) has \(\pn^X\) and the action representation is coefficient-separating, then \(C_X^*G\) is simple.
\end{theorem}

This approach is a natural step, since it captures the main algebraic ideas behind the averaging arguments without yet dealing with the topology of the ambient group.

\vspace{2mm}

Our main geometric application concerns mapping class group of infinite type surface $S$ containing a nondisplaceable subsurface of finite type. Horbez-Qing-Rafi \cite{HQR20} proved that if $S$ contains a nondisplaceable subsurface of finite type, then there exists a continuous, nonelementary, isometric action of \(\operatorname{Map}(S)\) on a hyperbolic space \(\mathbb{X}\), constructed using the projection complex machinery of \cite{BBF15}. We show that this construction can be fixed so that such big mapping class group satisfies property \(\pn^{\mathbb X}\).

\begin{theorem}[See Theorem~\ref{TL1}] \label{thm:intro-TL1}
Let \(S\) be a connected orientable infinite type surface without boundary containing a connected nondisplaceable finite type subsurface. Then there exists a fixed hyperbolic \(\mathbb X\) such that \(\MapS\) has property \(\pn^{\mathbb X}\). 
\end{theorem}

The idea of the proof of Theorem \ref{thm:intro-TL1} can be summarized as follows: Using Lemma~\ref{lem:clean-enlargement}, fix once and for all a connected clean nondisplaceable finite type subsurface \(K\subset S\) with \(\xi(\widehat K)\ge3\), in the sense specified in Section \ref{5}. Let \(\mathbf Y_K=\MapS\cdot K\), equipped with the subsurface projections introduced by Horbez-Qing-Rafi \cite{HQR20}. Apply the strong-projection modification of Bestvina-Bromberg-Fujiwara-Sisto \cite{BBFS} equivariantly, fix the projection data \(\pi^{\prime}\) and a sufficiently large number \(\alpha\), and let \(\mathbb X=\mathcal C_\alpha(\mathbf Y_K;\pi^{\prime})\) be the associated quasi-tree of curve graphs. These choices are made once and are independent of the finite subset $F=\{h_{1}, \cdots, h_{n}\}\subseteq \MapS$ considered below. The geometric difficulty is to control all nonidentity powers of finitely many mapping classes $h_i \in F$ for $1 \leq i\leq n$ while keeping the action of $\MapS$ on $\mathbb{X}$ fixed. We address it using a finite type subsurface containing curves that detect those powers. We then construct a WWPD loxodromic element \(f\) whose endpoint stabilizers preserve its finite type support \(M\). By conjugating \(f\) by a suitable mapping class supported on \(M\), we obtain an element \(g\) such that no nonidentity power of any \(h\in F\) sends either endpoint of \(g\) to either endpoint of \(g\). Using the endpoint separation and the bounded-projection characterization of WWPD, we construct periodic quasi-geodesics in \(\mathbb X\) by joining long geodesic segments whose endpoints lie on suitable translates of a quasi-axis of \(g\). This construction shows that, for a common sufficiently large exponent \(Q\) and every \(h\in F\), the subgroup \(\langle h,g^Q\rangle\) is the free product \(\langle h\rangle*\langle g^Q\rangle\), and every element of this subgroup not conjugate into \(\langle h\rangle\) is loxodromic. Consequently, every elliptic subgroup of \(\langle h,g^Q\rangle\) is conjugate into \(\langle h\rangle\), completing the proof of property \(P_{\mathrm{naive}}^{\mathbb X}\).
\vspace{2mm}

The action of \(\MapS\) on \(\mathbb X\) is faithful by Lemma~\ref{a}. Proposition~\ref{P1} therefore gives the corresponding analytic property for this same action.

\begin{corollary}[See Corollary~\ref{C1}]
Let \(S\) be a connected orientable infinite type surface without boundary containing a connected nondisplaceable finite type subsurface. Then there exists a fixed hyperbolic \(\mathbb X\) such that \(\MapS\) has property \(\pa^{\mathbb X}\).
\end{corollary}

Although \(\MapS\) has property \(\pa^{\mathbb X}\), this property does not imply the simplicity of \(C_{\mathbb X}^*(\MapS)\), since the associated action representation is not coefficient-separating. Indeed, a vertex stabilizer contains a nonabelian free subgroup supported outside the corresponding finite type subsurface. Proposition~\ref{simple555} shows that, under the hypotheses of Corollary~\ref{C1}, $C_{\mathbb X}^*\operatorname{Map}(S)$ is not simple, although it has a unique normalized trace.

\subsection*{Acknowledgments:}

I am deeply grateful to my advisors, Indira Chatterji and Fran\c{c}ois Dahmani, for their guidance, encouragement, and for many illuminating discussions and insightful comments. Thanks to the China Scholarship Council for their fellowship support.

\section{Preliminaries}

\subsection{Operators on Hilbert space} Let us now recall more details about the operators on Hilbert space \(\mathcal{H}\).

The inner product of two vectors \(\xi, \eta\) in a Hilbert space \(\mathcal{H}\) is denoted by \(\langle\xi, \eta\rangle\). We denote by \(\mathcal{L}\left(\mathcal{H}_{1}, \mathcal{H}_{2}\right)\) the vector space of all continuous linear operators from a Hilbert space \(\mathcal{H}_{1}\) to another Hilbert space \(\mathcal{H}_{2} \) and an operator \(T \in \mathcal{L}\left(\mathcal{H}_{1}, \mathcal{H}_{2}\right)\) has an adjoint \(T^{*} \in \mathcal{L}\left(\mathcal{H}_{2}, \mathcal{H}_{1}\right)\). We write \(\mathcal{L}(\mathcal{H})\) for \(\mathcal{L}(\mathcal{H}, \mathcal{H})\) and observe that it is naturally an involutive complex algebra with unit operator \(I\) of \(\mathcal{H}\). An operator \(U: \mathcal{H} \rightarrow \mathcal{H}\) is \emph{unitary} if $$ U U^{*}=U^{*} U=I $$ or, equivalently, if \(\langle U \xi, U \eta\rangle=\langle\xi, \eta\rangle\) for all \(\xi, \eta \in \mathcal{H}\) and if \(U\) is onto. The unitary group \(\mathcal{U}(\mathcal{H})\) of \(\mathcal{H}\) is the group of all unitary operators in \(\mathcal{L}(\mathcal{H})\).

A linear map $T:\mathcal{H}_{1}\to\mathcal{H}_{2}$ is called \emph{bounded} if and only if there exists some $M>0$ such that for all $x \in \mathcal{H}_{1}$,
\[
\|Tx\|_{\mathcal{H}_{2}} \leq M\|x\|_{\mathcal{H}_{1}}.
\]
The smallest such $M$ is called the \emph{operator norm} of $T$ and denoted by $\|T\|_{\mathrm{op}}$. A linear operator between normed spaces is continuous if and only if it is bounded. We denote by
\[
\mathcal{B}(\mathcal{H}_{1},\mathcal{H}_{2})=\{T:\mathcal{H}_{1}\to\mathcal{H}_{2}\mid T \text{ is linear and bounded}\}
\]
the Banach space of all bounded operators, equipped with the operator norm $\|\cdot\|_{\mathrm{op}}$. When $\mathcal{H}_{1}=\mathcal{H}_{2}=\mathcal{H}$ we write $\mathcal{B}(\mathcal{H})$.

\subsection{Action representations} The Koopman's framework recasts dynamical systems as operator-theoretic objects. Instead of tracking how a group $G$ moves points of a measure space $(X, \mu)$, the Koopman's viewpoint follows how the action pulls back square-summable observables. This turns the nonlinear dynamics of the action of $G$ on $X$ into a linear, unitary representation on the Hilbert space $L^{2}(X, \mu)$.

\begin{definition}
 Let $G$ be a group acting measure-preservingly on a countable metric space $X$ endowed with counting measure $\mu$ and $\ell^{2}(X)$ be a Hilbert space. For each $g\in G$ define $$ \left(\pi(g) f\right)(x):=f\left(g^{-1} \cdot x\right), \quad f \in \ell^{2}(X). $$ Because the action preserves \(\mu\), every \(\pi(g)\) is unitary, and the map $$ \pi: G \longrightarrow \mathcal{U}(\ell^{2}(X)), \quad g \mapsto \pi(g), $$ is a group homomorphism. This homomorphism \(\pi\) is called the \emph{action representation} associated with the action of $G$ on $X$. 
 \end{definition}

The action representation lifts this group action to the Hilbert space $\ell^{2}(X)$. The representation \(\pi\) maps \(g\) to a unitary operator on \(\ell^{2}(X)\) because for all \(\xi, \eta \in \ell^{2}(X)\), one has
\[\|\pi(g) \xi\|^{2}=\sum_{x \in X}\left|\xi\left(g^{-1} x\right)\right|^{2}=\sum_{y \in X}|\xi(y)|^{2}=\|\xi\|^{2},\] and 
\[\langle\pi(g) \xi, \pi(g) \eta\rangle=\langle\xi, \eta\rangle.\]

\begin{lemma}\label{f}
Let $G$ be a group acting on a countable metric space $X$ by isometries and let $\pi: G \rightarrow \mathcal{U}(\ell^2(X))$ be the associated action representation. Then the action of $G$ on $X$ is faithful if and only if $\pi$ is faithful.
\end{lemma}

\begin{proof}
($\Leftarrow$) If the action of $G$ on $X$ is not faithful, there exists $g\neq e$ with $g\cdot x=x$ for all $x\in X$. Then for every $\xi\in \ell^2(X)$ and $x\in X$,
\[
(\pi(g)\,\xi)(x)=\xi(g^{-1}\!\cdot x)=\xi(x),
\]
hence $\pi(g)=\mathrm{Id}$ and $\pi$ is not faithful.

($\Rightarrow$) Conversely, suppose $\pi$ is not faithful, so there exists $g\neq e$ with $\pi(g)=\mathrm{Id}$. Fix any $x_0\in X$ and consider the Dirac vector $\delta_{x_0}\in \ell^2(X)$ satisfying $\delta_{x_{0}}(x)=1$ if $x=x_{0}$, else $0$. Then \( \pi(g)\,\delta_{x_0}=\delta_{g\!\cdot x_0}\). Since $\pi(g)=\mathrm{Id}$, we have $\delta_{g\!\cdot x_0}=\delta_{x_0}$, hence $g\cdot x_0=x_0$. As $x_0$ was arbitrary, $g$ fixes every point of $X$, so the action is not faithful.
\end{proof}

\subsection{Group algebra of \texorpdfstring{$G$}{G}}
Let us recall the relationship between the space of operators and group algebra. The group algebra of $G$, denoted by \(\mathbb{C} G\), is defined to be the set of all formal finite sums \(\sum_{g\in G} a_{g} g\), where \(a_{g} \in \mathbb{C}\) and \(a_{g} =0\) with finitely many exceptions, together with the rules of addition and multiplication
$$
\bigl(\sum_{g\in G} a_{g} g\bigr)+ \bigl(\sum_{g\in G} a_{g}^{\prime} g\bigr)=\sum_{g\in G} (a_{g}+a_{g}^{\prime})g
$$
and 
$$
\bigl(\sum_{g\in G} a_{g} g\bigr)\bigl(\sum_{g\in G} a_{g}^{\prime} g\bigr)=\sum_{g\in G}\bigl(\sum_{hk=g}a_{h}a_{k}^{\prime}\bigr)g.
$$

The action representation of $G$ can be extended from the group to the group algebra by 
$$
\pi:\mathbb{C} G \rightarrow \mathcal{L}(\ell^{2}(X)),  \pi\bigl(\sum a_{g} g\bigr)=\sum a_{g} \pi(g).
$$

The next lemma shows that this extension indeed lands in the space of bounded operators.

\begin{lemma} \label{b} Let $G$ be a group acting on a countable metric space $X$ by isometries and let $\pi: G \rightarrow \mathcal{U}(\ell^2(X))$ be the associated action representation. Then the representation $\pi$ of $\mathbb{C} G$ on $\ell^{2}(X)$ is by bounded operators, i.e. $\pi: \mathbb{C} G \rightarrow \mathcal{B}(\ell^{2}(X))$.  
\end{lemma}
\begin{proof}
    For any $\xi \in \ell^{2}(X)$, by the isometric action of \(G\) on \(X\), we have
    $$\|\pi(g) \xi\|^{2}=\sum_{x \in X}\left|\xi\left(g^{-1} x\right)\right|^{2}=\sum_{y \in X}|\xi(y)|^{2}=\|\xi\|^{2}.$$
    Thus, 
    $$\|\pi(g) \|_{\mathrm{op}}=\operatorname{sup}_{\|\xi\|=1}\|\pi(g) \xi\|=\operatorname{sup}_{\|\xi\|=1}\| \xi\|=1.$$
    For any $\sum a_{g} g \in \mathbb{C} G$, we have
    $$\left\|\pi\left(\sum a_{g} g \right) \right\|_{\mathrm{op}} \leq \sum |a_{g}|\|\pi(g)\|_{\mathrm{op}}=\sum |a_{g}|.$$
    Therefore, $\pi(z) \in \mathcal{B}(\ell^{2}(X))$ for every $z \in \mathbb{C} G$.
\end{proof}
\begin{remark}
For every finite subset $F=\{g_{0}, g_{1}, \cdots, g_{n}\}
\subseteq G$, assume that there exists $x \in X$ such that $g_{i}(x)$ are pairwise distinct for all $i \in \{0, \cdots, n\}$. Then $\pi: \mathbb{C} G \rightarrow \mathcal{B}(\ell^{2}(X))$ is faithful.

 Let $0\ne a=\sum_{i=0}^{n}a_i g_i\in\mathbb CG$, where the $g_i$ are
distinct. By the finite-separation hypothesis, choose $x\in X$ such
that $g_0x,\ldots,g_nx$ are pairwise distinct. Then
\[
 \pi(a)\delta_x=\sum_{i=0}^{n}a_i\delta_{g_ix},
\]
and the vectors on the right are pairwise orthogonal. Consequently,
\[
 \|\pi(a)\delta_x\|^2=\sum_{i=0}^{n}|a_i|^2>0.
\]
Thus $\pi(a)\ne0$, and the extension
$\pi:\mathbb CG\to\mathcal B(\ell^2(X))$ is injective.
\end{remark}

\subsection{Action-based \texorpdfstring{$C^{*}$}{C*}-algebra}
Let $G$ be a group acting on a countable metric space $X$. Assume that the action is faithful.
The \emph{action-based $C^{*}$-algebra} of $G$ relative to $X$, denoted by \(C_{X}^{*} G\), is the norm closure of the linear span of the set of operators $\{\pi(g): g \in G\}$ in $\mathcal{B}(\ell^{2}(X))$, namely
$$
C_{X}^{*} G=\overline{\operatorname{Span}\{\pi(g):g \in G\}}{ }^{\|\cdot\|_{\text { operator }}},
$$
where $\|\cdot\|_{\mathrm{op}}$ denotes the operator norm on $\mathcal{B}\left(\ell^{2}(X)\right)$, given by 
$$\|T\|_{\mathrm{op}}=\operatorname{sup}\left\{\|T \xi\| | \xi \in \ell^{2}(X), \|\xi\|=1 \right\},$$
where $\|\xi\|=\sqrt{\sum_{x \in X}|\xi(x)|^{2}}$ is the norm in $\ell^{2}(X)$.

\begin{lemma}
The algebra \(C_X^*G\) is a unital \(C^*\)-subalgebra of \(\mathcal B(\ell^2(X))\).
\end{lemma}

\begin{proof}
Let \(A=\operatorname{Span}\{\pi(g):g\in G\}\). Since \(\pi\) is a unitary representation, for every \(g,h\in G\),
\[
\pi(g)\pi(h)=\pi(gh),
\quad
\pi(g)^*=\pi(g^{-1}),
\quad
\pi(e)=\mathrm{Id}_{\ell^2(X)}.
\]
It follows that \(A\) is a unital \(*\)-subalgebra of \(\mathcal B(\ell^2(X))\). Therefore its operator-norm closure
\( C_X^*G=\overline{A}^{\|\cdot\|_{\mathrm{op}}} \) is a unital \(C^*\)-subalgebra of \(\mathcal B(\ell^2(X))\).
\end{proof}




\section{Properties for groups acting on a countable metric space}

\subsection{Generalized analytic property \texorpdfstring{\(\pa^{X}\)}{P{analytic}{X}}}

Building on \cite[Definition 1]{BCD95}, we think the following notion is natural to consider.

\begin{definition}\label{ana1}
Let $G$ be a group acting on a countable metric space $X$ by isometries and let $\pi: G \rightarrow \mathcal{U}(\ell^2(X))$ be the associated action representation. We say that \(G\) satisfies \emph{property \(\pa^X\)} if for every finite subset \(F \subset G \setminus \{1\}\), there exist an element \(g \in G\) and a constant \(C > 0\) such that for any \(h \in F\) and any $a \in \ell^2(\mathbb{Z}^+)$, we have:
\[
\left\| \sum_{j=1}^{\infty} a_j \pi(g^{-j} h g^j) \right\|_{\mathrm{op}} \leq C \left\| a \right\|_{\ell^{2}}.
\]
Here $a_{j}$ is the $j$th term of the sequence $a$ and $\mathbb{Z}^{+}$ denotes the set of positive integers.
\end{definition}
The series is understood as the operator-norm limit of its finite partial sums; equivalently, it is enough to require the estimate for finitely supported sequences, after which the general case follows by completion.

\begin{remark}
\begin{itemize}
\item It is important that Definition~\ref{ana1} is formulated relative to the action representation \(\pi\), and not with respect to all unitary representations of \(G\). Indeed, the universal version would already fail for the trivial representation \(1_G\), since in that case
\[
\left\|\sum_{j=1}^{\infty} a_j\,1_G(g^{-j}hg^j)\right\|
=
\left|\sum_{j=1}^{\infty} a_j\right|,
\]
which is not bounded by \(C\|a\|_{\ell^2}\) for all \(a\in \ell^2(\mathbb Z_+)\). Also, a formulation involving all unitary representations would already exclude every nontrivial finite group. Indeed, if \(G\) is finite and \(g\in G\), then \(g\) has finite order, say \(g^m=e\). Hence for all \(j\ge 1\), \(g^{-(j+m)}hg^{j+m}=g^{-j}hg^j\). So the sequence of conjugates \(g^{-j}hg^j\) is periodic. Applying any unitary representation \(\rho\), the sequence $\rho(g^{-j}hg^j)$ is also periodic. Now choose coefficients by
\[a_{1+km}=\frac1{\sqrt N}\quad (0\le k\le N-1),\quad a_j=0\text{ otherwise}.\]
Then \(\|a\|_{\ell^2}=1\), while
\[
\sum_{j\ge 1} a_j\,\rho(g^{-j}hg^j)
=
\sum_{k=0}^{N-1}\frac1{\sqrt N}\,\rho(g^{-1}hg)
=
\sqrt N\,\rho(g^{-1}hg).
\]
Since \(\rho(g^{-1}hg)\) is unitary, its operator norm is equal to \(1\). Thus
\[
\left\|\sum_{j\ge 1} a_j\,\rho(g^{-j}hg^j)\right\|_{\mathrm{op}}
=
\sqrt N.
\]
As \(N\) is arbitrary, no constant \(C>0\) can satisfy the required estimate.

\item Whenever $G$ is countable, we can take the left translation action of $G$ on itself and choose any left-invariant metric on $G$. If we set $X=G$ and $G$ acts on it by left multiplication, then the action representation $\pi$ on $\ell^{2}(G)$ equals the left-regular representation $\lambda$ of $G$. Therefore, \(\pa^X=\pa\) (see \cite{BCD95}).
    \item If $X=\{x\}$, only $G=\{e\}$ satisfies $\pa^{X}$.
\end{itemize}
\end{remark}

\begin{definition}\label{naive1}
    A group $G$ acting on a countable hyperbolic metric space $X$ is said to have \emph{property $\pn^X$} if, for any finite subset $F$ of $G \setminus \{1\}$, there exists an element \(g \in G\) which is loxodromic for the action of $G$ on $X$ such that, for every \(h \in F\), the subgroup \(\left\langle h, g\right\rangle\) of \(G\) is isomorphic to the free product \(\langle h\rangle *\left\langle g\right\rangle\) and the only elliptic subgroups of \(\left\langle h, g\right\rangle\) are conjugate into \(\langle h\rangle\).
\end{definition}
Here and below, a subgroup is called \emph{elliptic} if each of its elements is elliptic for the action under consideration.

The following proposition is a generalization of \cite[Lemma 2.2]{BCD95}, that is, if a group $G$ has the property $\pn^X$, then it has the property \(\pa^X\). 
\begin{proposition} \label{P1}
Let $G$ be a group acting faithfully on a countable hyperbolic metric space $X$ by isometries and assume $G$ has the property $\pn^{X}$. Then $G$ satisfies the property \(\pa^X\).
\end{proposition}

We now fix the following notation and use it throughout this subsection. Assume $G$ has the property $\pn^{X}$. Fix a finite set $F\subset G\setminus\{e\}$ and choose $g\in G$ loxodromic for the action of $G$ on $X$ given by $\pn^{X}$ so that for each $h\in F$,
\[
H:=\langle h,g\rangle\ \cong\ \langle h\rangle * \langle g\rangle .
\]
We work with the normal form on $H$. Let $W_0\subset H$ be the set of reduced words whose first syllable lies in $\langle h\rangle$ (equivalently: whose first letter is not a nontrivial power of $g$), and for $j\in\mathbb Z$, set $W_j:=g^{\,j}W_0$.

Before proving this proposition, we first prove the following lemma with the above setup.

\begin{lemma}\label{free1}
Let $G$ be a group acting on a countable hyperbolic metric space $X$ by isometries and assume $G$ has the property $\pn^{X}$. 
Then there exists a countable set $\{x_i\}_{i\ge0}\subset X$ such that, for each $x_i$, $0\leq i< \infty$, and $j\neq k$,
\((W_j\cdot x_i)\ \cap\ (W_k\cdot x_i)=\varnothing.\)
\end{lemma}

\begin{proof}
Fix $h \in F$ and let $H:=\langle h\rangle * \langle g\rangle$ as above. Choose a base point $x_{0}\in X$, we write the orbit $$O_{0}=H \cdot x_{0}=\{w \cdot x_{0} | w \in H\} \subset X.$$ Observe that for each $w\in H$, $w(O_{0}) \subset O_{0}$ and $w(X\setminus O_{0}) \subset X\setminus O_{0}$. Then we choose $x_{1}\in X\setminus O_{0}$ and write the orbit $$O_{1}=H \cdot x_{1}=\{w \cdot x_{1} | w \in H\} \subset X.$$ We also have that $w(O_{1}) \subset O_{1}$ and $w(X\setminus (O_{0}\cup O_{1})) \subset X\setminus (O_{0}\cup O_{1})$. By induction, we have that $X=\bigcup_{i=0}^{\infty}O_{i}=\bigcup_{i=0}^{\infty}H\cdot x_{i}$, where $\{x_0, x_1, \cdots\}$ is a countable set.

Since \(X\) is countable, the set of \(H\)-orbits in \(X\) is finite or countable. Choose one representative \(z_i\) from each \(H\)-orbit and let \(K_i=\operatorname{Stab}_H(z_i)\). Each \(K_i\) fixes \(z_i\), and hence is elliptic. By assumption, there exists \(a_i\in H\) such that \(a_i^{-1}K_i a_i\leq \langle h\rangle\).
Write \(x_i=a_i^{-1}\cdot z_i\). Since \(x_i\) belongs to the same \(H\)-orbit as \(z_i\), the collection \(\{x_i\}_{i\in I}\) is still a complete set of orbit representatives. Moreover, \(\operatorname{Stab}_H(x_i)=a_i^{-1}K_i a_i\leq\langle h\rangle\).

Suppose \((W_j\cdot x_i)\ \cap\ (W_k\cdot x_i)\neq \varnothing\), then there are $u,v\in W_0$ with $(g^j u)\cdot x_i=(g^k v)\cdot x_i$. Then $v^{-1}g^{\,j-k}u\in K_i$. Because the first letters of $u$ and $v$ are outside $\langle g\rangle$, the reduced form of $v^{-1}g^{\,j-k}u$ (with $j\neq k$) necessarily contains a non-trivial power of $g$, hence is not into $\langle h\rangle$. This contradicts the notion that the stabilizer of $x_{i}$ is in $\langle h\rangle$.
Therefore, for each $i$, the sets $W_j\cdot x_i$, for $j \in \mathbb{Z}$, are pairwise disjoint.
\end{proof}

Then, we give the proof of Proposition \ref{P1} following \cite[Lemma 2.2]{BCD95}.

\begin{proof}[Proof of Proposition \ref{P1}.]
Let \(F = \{h_1, \dots, h_n\} \subset G \setminus \{1\}\) be a finite subset. By Definition \ref{naive1}, there exists a loxodromic element \(g \in G\) such that for all \(i\), the subgroup \(\langle h_i, g \rangle \cong \langle h_i \rangle * \langle g \rangle\), that is, a free product.

Fix $h$ in $F$, let $H$ to be the subgroup $\langle h, g \rangle$ of $G$. Denote by $W_{0}$ the subset of  $H$ consisting of the words which do not begin with a nontrivial power of $g$ and by $W_{j}$ the set $g^{j}W_{0}$, for all $j \in \mathbb{Z}$. By Lemma \ref{free1}, there exists a countable base ${x_0,x_1,\ldots}$ such that, for each $i\geq 0$, the family $\{W_j^i\}_{j\in\mathbb Z}$ associated to $x_i$ satisfies that the sets $W_j^i\cdot x_i$, $j\in\mathbb Z$, are pairwise disjoint.
We also write $O_{i}=H \cdot x_{i}=\{w \cdot x_{i} | w \in H\} \subset X$. So we can decompose the Hilbert space $\ell^{2}(X)=\bigoplus_{i}\ell^{2}(O_i)$. 

Denote $\pi$ by the action representation of $H$ and define $\chi_{W_{j}^i}: \ell^{2}(O_i) \rightarrow \ell^{2}(O_i)$ to be a bounded projection operator on $\ell^{2}(O_i)$ by the indicator $W_{j}^i\cdot x_{i}$, for all $x \in O_i$, that is, 
$$
\chi_{W_{j}^i}(\xi^{i})(x)=\begin{cases}
    \xi^{i}(x), & \mbox{if}~x \in W_{j}^i \cdot x_{i},\\
    0, & \mbox{otherwise}.
\end{cases}
$$
(For all $x\in O_i$ and $\xi^i \in \ell^2(O_i)$, 
$$\|\chi_{W_{j}^i}\xi^i\|^{2}=\sum_{x \in O_i}|\chi_{W_{j}^i}\xi^i(x)|^{2}\leq \sum_{x \in O_i}|\xi^i(x)|^{2}=\|\xi^i\|^{2},$$ so it is bounded. Moreover $\chi_{W_{0}}^{2}=\chi_{W_{0}}$, so it is a projection.
)

For $\xi, \eta \in \ell^2(X)$, one has $\xi=\sum_{i}\xi^{i}$ and $\eta=\sum_{i}\eta^i$. 
For $j \in \mathbb{Z}$, since $\pi(gh)=\pi(g)\pi(h)$ and $\pi(g^{-j})\pi(g^{j})=I$, we obtain
\begin{align*}
    |\langle\pi(g^{-j}hg^{j})\xi, \eta \rangle|&=\left|\sum_{i}\langle\pi(g^{-j}hg^{j})\xi^i, \eta^i  \rangle\right|=\left|\sum_{i}\langle\pi(hg^{j})\xi^i, \pi(g^{j})\eta^i \rangle\right|\\& =\left|\sum_i\langle\pi(h)(\chi_{W_{0}^i}\pi(g^{j})\xi^i+\chi_{H\setminus W_{0}^i}\pi(g^{j})\xi^i), \pi(g^{j})\eta^i \rangle\right|\\
    &\leq \left|\sum_i\langle\pi(h)\chi_{W_{0}^i}\pi(g^{j})\xi^i, \pi(g^{j})\eta^i \rangle\right| +\left|\sum_i\langle\pi(h)\chi_{H\setminus W_{0}^i}\pi(g^{j})\xi^i, \pi(g^{j})\eta^i \rangle\right|
\end{align*}

Observe that, for all $x \in O_i$, $\chi_{W_{\alpha}^i}\pi(g^{j})\xi^i(x)=\pi(g^{j})\chi_{W_{\alpha-j}^i}\xi^i(x)$. Then, using the Cauchy-Schwarz inequality and $h(H\setminus W_{0}^i)\subset W_{0}^i$, we have
\begin{align*}
    |\langle\pi(g^{-j}hg^{j})\xi, \eta \rangle|&\leq \left|\sum_i\langle\pi(h)\pi(g^{j})\chi_{W_{-j}^i}\xi^i, \pi(g^{j})\eta^i \rangle\right|+\left|\sum_i\langle\pi(h)\chi_{H\setminus W_{0}^i}\pi(g^{j})\xi^i, \chi_{W_{0}^i}\pi(g^{j})\eta^i \rangle\right|\\
    & \leq \left\|\sum_i\chi_{W_{-j}^i}\xi^i\right\|\left\|\sum_i\eta^i\right\|+\left\|\sum_i \xi^i\right\|\left\|\sum_i\chi_{W_{-j}^i}\eta^i\right\|.
\end{align*}

Then take $a \in \ell^2(\mathbb{Z}^{+})$, where $a: \mathbb{Z}^{+} \rightarrow \mathbb{C}, j \mapsto a(j):=a_{j}$, and define the operator $T_{a}$ on $\ell^{2}(X)$ by 
$$T_{a}= \sum_{j=1}^{\infty}a_{j}\pi(g^{-j}hg^{j}).$$
Then, for any $\xi, \eta \in \ell^{2}(X)$, using Cauchy-Schwarz inequality, we have
\begin{align*}
 |\langle T_{a}\xi, \eta\rangle| &\leq  \sum_{j=1}^{\infty}|a_{j}|\left(\left\|\sum_i\chi_{W_{-j}^i}\xi^i\right\|\left\|\eta\right\|+\left\|\xi\right\|\left\|\sum_i\chi_{W_{-j}^i}\eta^i\right\|\right)\\&\leq \|\eta\|\left(\sum_{j=1}^{\infty}|a_{j}|^2\right)^{\frac{1}{2}}\left(\sum_{j=1}^{\infty}\left\|\sum_i\chi_{W_{-j}^i}\xi^i\right\|^2\right)^{\frac{1}{2}}+ \|\xi\|\left(\sum_{j=1}^{\infty}|a_{j}|^2\right)^{\frac{1}{2}}\left(\sum_{j=1}^{\infty}\left\|\sum_i\chi_{W_{-j}^i}\eta^i\right\|^2\right)^{\frac{1}{2}}\\&\leq 2\|a\|_{\ell^2}\|\xi\|\|\eta\| .   
\end{align*}

By the Riesz representation theorem and the equivalent definitions of the operator norm, we obtain
$$\|T_{a}\xi\|=\operatorname{Sup}_{\|\eta\|=1}|\langle T_{a}\xi, \eta\rangle|\leq 2\|a\|_{\ell^2}\|\xi\|,$$
Then $$\|T_{a}\|_{\mathrm{op}}=\operatorname{sup}_{\|\xi\|=1}\|T_{a}\xi\|\leq\sup_{\|\xi\|=1}2\|a\|_{\ell^2}\|\xi\|=2\|a\|_{\ell^2}.$$
\end{proof}

\section{The simplicity of the action-based \texorpdfstring{$C^{*}$}{C*}-algebra}

\subsection{A conjugation averaging argument}

Next, we will introduce the simplicity of the action-based $C^*$-algebra. The argument of Bekka-Cowling-de la Harpe \cite{BCD95} cannot be directly adapted to our setting. Their weak-containment method is tailored to the left regular representation, where the diagonal coefficient of the Dirac vector at the identity is $\delta_e$. For the action representation $\pi:G\to \mathcal{U}(\ell^2(X))$, however, one has
\[
\langle \pi(h)\delta_x,\delta_x\rangle=
\begin{cases}
1,& h\cdot x=x,\\
0,& h\cdot x\neq x,
\end{cases}
\]
so this coefficient detects $\operatorname{Stab}(x)$ rather than $\{e\}$. Since the action on $X$ is not free in our applications, we cannot use their approach and instead argue by conjugation averaging. In this subsection, we prove the simplicity of the action-based \(C^*\)-algebra from property \(P^X_{\mathrm{analytic}}\) together with the coefficient-separating assumption by a Powers-type averaging argument.

For $x\in C_X^*G$ and $g\in G$, we consider the conjugates \(\pi(g^{-1})\, x\, \pi(g)\), and their averages. The key point is that property $\pa^{X}$ implies that the nontrivial group terms disappear under such averages, while the scalar part remains unchanged.

\begin{lemma}\label{lem:averaging-finite-sum}
Let $G$ act faithfully on a countable metric space $X$, and let \(\pi:G\to \mathcal{U}(\ell^2(X)) \) be the associated action representation. Assume that $G$ has property $\pa^{X}$. Let \( T=a_e Id+\sum_{h\in F} a_h \pi(h),\) where $F\subseteq G\setminus\{e\}$ is finite. Then there exists $g\in G$ such that
\[
\frac1J\sum_{j=1}^J \pi(g^{-j})T\pi(g^j)\longrightarrow a_e I
~\text{ in operator norm as $J \to \infty$.}
\]
\end{lemma}

\begin{proof}
Since $G$ has property $\pa^{X}$, for the finite subset $F\subseteq G\setminus\{e\}$ there exist an element $g\in G$ and a constant $C>0$ such that for every $h\in F$ and every $a=(a_j)_{j\ge1}\in \ell^2(\mathbb Z^+)$, one has
\[
\left\|\sum_{j\ge1} a_j\,\pi(g^{-j}hg^j)\right\|
\le C\|a\|_{\ell^2}.
\]

Fix $J\ge1$, and define a sequence $a=(a_j)_{j\ge1}$ by
\[
a_j=
\begin{cases}
\frac1J,& 1\le j\le J,\\[4pt]
0,& j>J.
\end{cases}
\]
Then
\[
\|a\|_{\ell^2}
=
\left(\sum_{j=1}^\infty |a_j|^2\right)^{1/2}
=
\left(J\cdot \frac1{J^2}\right)^{1/2}
=
\frac1{\sqrt J}.
\]
Hence, for every $h\in F$,
\[
\left\|
\frac1J\sum_{j=1}^J \pi(g^{-j}hg^j)
\right\|
\le
\frac{C}{\sqrt J}.
\]

Now expand the conjugation average of $T$:
\begin{align*}
&\frac1J\sum_{j=1}^J \pi(g^{-j})T\pi(g^j)=
\frac1J\sum_{j=1}^J \pi(g^{-j})
\left(a_e Id+\sum_{h\in F} a_h\pi(h)\right)\pi(g^j)\\&= a_e\cdot \frac1J\sum_{j=1}^J \pi(g^{-j})Id \pi(g^j)
+
\sum_{h\in F} a_h
\left(
\frac1J\sum_{j=1}^J \pi(g^{-j})\pi(h)\pi(g^j)
\right)\\&=a_e Id+
\sum_{h\in F} a_h
\left(
\frac1J\sum_{j=1}^J \pi(g^{-j}hg^j)
\right).
\end{align*}

Applying the triangle inequality for the operator norm,
\begin{align*}
   \left\|
\frac1J\sum_{j=1}^J \pi(g^{-j})T\pi(g^j)-a_e Id
\right\|
&\le
\sum_{h\in F}|a_h|
\left\|
\frac1J\sum_{j=1}^J \pi(g^{-j}hg^j)
\right\|  \le
\frac{C}{\sqrt J}\sum_{h\in F}|a_h|.
\end{align*}
Since the right-hand side tends to $0$ as $J\to\infty$, the conclusion follows.
\end{proof}

Let $A:=\operatorname{span}\{\pi(g):g\in G\}$. For every $k\in G$ and finite subset $F_{0} \subset G$, the \emph{coefficient map} is defined by
\[
E_k:A\to \mathbb C,\quad
E_k\left(\sum_{g \in F_{0}}a_g\pi(g)\right)=a_k.
\]

\begin{lemma}\label{lem:coefficient-functionals}
Assume that $G$ satisfies property $\pa^X$. Then the coefficient map $E_k$ for every $k \in G$ is well-defined and satisfies $|E_k(T)|\leq \|T\|$ for $T\in A$.
\end{lemma}

\begin{proof}
For $g\in G$ and $J\geq 1$, define
\[
M_{J,g}(y):=\frac1J\sum_{j=1}^J
\pi(g^{-j})\,y\,\pi(g^j),
\quad y\in C_X^*G.
\]
Since each map $y\mapsto \pi(g^{-j})y\pi(g^j)$ is an isometric automorphism of $C_X^*G$, the averaging map $M_{J,g}$ is contractive and
\(
\|M_{J,g}(y)\|\leq \|y\|.
\)

For any finite subset $F_{0} \subset G$, let $T=\sum_{g \in F_{0}}a_g\pi(g)\in A$. Fix $k\in G$. Then
\[
\pi(k^{-1})T
=a_k Id+\sum_{h\in F}a_h\pi(h)
\]
where $F=k^{-1}F_0\setminus\{e\}$.

By Lemma \ref{lem:averaging-finite-sum}, there exists some $g\in G$ such that $M_{J,g}\bigl(\pi(k^{-1})T\bigr) \rightarrow a_k Id$ in operator norm as $J\to\infty$. Therefore,
\[
|a_k|
=
\|a_k Id\|
=
\lim_{J\to\infty}
\left\|M_{J,g}\bigl(\pi(k^{-1})T\bigr)\right\|
\leq
\|\pi(k^{-1})T\|
=
\|T\|.
\]
Thus $|E_k(T)|\leq \|T\|$. In particular, if $T=0$ as an operator, then $a_k=0$ for every $k\in G$. Hence, the coefficient $a_k$ is uniquely determined by the operator $T$, so $E_k$ is well-defined on $A$.
\end{proof}

Since $E_k$ is linear and bounded on the dense subspace $A\subset C_X^*G$, it extends uniquely to a bounded linear functional on $C_X^*G$, still denoted by $E_k$, with $\|E_k\|\leq 1$.

Under property $\pa^X$, there is a canonical quotient from $C_X^*G$ onto $C_r^*(G_d)$, where $G_d$ denotes the underlying group equipped with the discrete topology. As shown in Proposition~\ref{prop:regular-quotient} below, coefficient separation is precisely injectivity of this quotient.

It asserts that no nonzero element of \(C_X^*G\) is invisible to all group coefficients. This is precisely the property needed in the averaging argument: starting from a nonzero element in an ideal, one must be able to extract a nonzero coefficient and then average the remaining nontrivial terms away.
\begin{definition}\label{coefficient-separating} 
We say that the action representation is \emph{coefficient-separating} if the family $\{E_k:C_X^*G\to\mathbb C~|~k \in G\}$ separates points, that is,
\[
\bigcap_{k\in G}\ker E_k=\{0\}.
\]
Equivalently, for every nonzero \(x\in C_X^*G\), there exists \(k\in G\) such that \(E_k(x)\neq0\).
\end{definition}

\begin{theorem}\label{thm:Pana-implies-simple-finite}
Assume that the action of a group $G$ acting on a countable metric space $X$ is faithful. Let $\pi$ be the associated action representation which is coefficient-separating. If $G$ has property $P^{X}_{\text{analytic}}$, then $C_X^*G$ is simple.
\end{theorem}

\begin{proof}
It is sufficient to show that every nonzero closed two-sided ideal of $C_X^*G$ is equal to $C_X^*G$. Let $I\subseteq C_X^*G$ be a nonzero closed two-sided ideal. Choose a nonzero element $x\in I$.    
Since the action representation is coefficient-separating, there exists $k\in G$ such that $E_k(x)\neq 0$. Set $y:=\pi(k^{-1})x$. Since $I$ is a two-sided ideal and $x\in I$, we have $y\in I$.
We claim that $E_e(y)=E_k(x)$.

Indeed, for every finite sum $T\in A$, one has directly from the definition of coefficients that $E_e(\pi(k^{-1})T)=E_k(T)$. Since $A$ is dense in $C_X^*G$, and since both sides are continuous in $T$, the same identity holds for every $T\in C_X^*G$. Applying this to $T=x$ gives $E_e(y)=E_k(x)\neq 0$.

Let $\varepsilon>0$.  By the density of $A$, choose $T\in A$ such that
\(
\|y-T\|<\frac{\varepsilon}{4}.
\)
By Lemma~\ref{lem:coefficient-functionals}, $\|E_e\|\leq 1$, and hence
\[
|E_e(y)-E_e(T)|
=
|E_e(y-T)|
\leq
\|y-T\|
<
\frac{\varepsilon}{4}.
\]

Write $$T=E_e(T)Id+\sum_{h\in F}a_h\pi(h),$$ where $F\subset G\setminus\{e\}$ is finite. By Lemma \ref{lem:averaging-finite-sum}, there exists $g\in G$ such that
\[
M_{J,g}(T)\longrightarrow E_e(T)Id
\]
in norm. Hence there exists $J_0\geq 1$ such that for all $J\geq J_0$,
\[
\|M_{J,g}(T)-E_e(T)Id\|<\frac{\varepsilon}{2}.
\]

Using $\|M_{J,g}\|\leq 1$, we obtain, for all $J\geq J_0$,
\[
\begin{aligned}
\|M_{J,g}(y)-E_e(y)Id\|&\leq
\|M_{J,g}(y-T)\|
+
\|M_{J,g}(T)-E_e(T)Id\|
+
\|E_e(T)Id-E_e(y)Id\|  \\
&\leq
\|y-T\|
+
\|M_{J,g}(T)-E_e(T)Id\|
+
|E_e(T)-E_e(y)|  \\
&<
\frac{\varepsilon}{4}
+
\frac{\varepsilon}{2}
+
\frac{\varepsilon}{4}
=
\varepsilon.
\end{aligned}
\]
It means that the scalar operator $E_e(y)Id$ lies in the norm closure of conjugation averages of $y$. 

Since $I$ is a two-sided ideal and $y\in I$, each term $\pi(g^{-j})\,y\,\pi(g^j)$ belongs to $I$. Hence $M_{J,g}(y)\in I$. Because $I$ is closed, it follows that $E_e(y)Id\in I$. Since $E_e(y)\neq 0$, we conclude that $Id\in  I$. Therefore $I=C_X^*G$.
\end{proof}

\begin{remark}\label{rem:coefficient-separating-necessary}
In the present setting, once the coefficient maps extend to \(C_X^*G\), the coefficient-separating assumption is necessary for simplicity. Indeed, write
\[
J:=\bigcap_{k\in G}\ker E_k .
\]
Then \(J\) is a closed subspace of \(C_X^*G\). Moreover, \(J\) is a two-sided ideal. To see this, let \(x\in J\) and \(s\in G\). For every \(k\in G\), using the relations on finite sums and then extending by continuity of $E_k$, we have
\[
E_k(\pi(s)x)=E_{s^{-1}k}(x)=0
\]
and
\[
E_k(x\pi(s))=E_{ks^{-1}}(x)=0.
\]
Hence \(\pi(s)x\in J\) and \(x\pi(s)\in J\) for every \(s\in G\). Since \(C_X^*G\) is the norm closure of the linear span of \(\{\pi(s):s\in G\}\), it follows that \(J\) is a closed two-sided ideal. Furthermore, \(J\) is proper, because
\[
E_e(1)=E_e(\pi(e))=1,
\]
so \(1\notin J\). Therefore, if the action representation is not coefficient-separating, then \(J\neq 0\), and \(J\) is a nonzero proper closed two-sided ideal of \(C_X^*G\). In particular, \(C_X^*G\) is not simple.
\end{remark}

We immediately have the following result.

\begin{corollary}\label{cor:Pnai-implies-simple-finite}
Assume that the action of a group $G$ acting on a countable hyperbolic metric space $X$ is faithful. Let $\pi$ be the associated action representation which is coefficient-separating. If $G$ has property $P^{X}_{\text{naive}}$, then $C_X^*G$ is simple.
\end{corollary}

\begin{proof}
By Proposition~\ref{P1}, $P^{X}_{\mathrm{naive}}$ implies $P^{X}_{\mathrm{analytic}}$.
Then apply Theorem~\ref{thm:Pana-implies-simple-finite}, we have the simplicity of $C_X^*G$.
\end{proof}

The following proposition will give the connection between the action-based \(C^*\)-algebra with the reduced \(C^*\)-algebra of the underlying discrete group.

\begin{proposition}\label{prop:regular-quotient}
Assume that $G$ has property $\pa^X$. Let \(G_d\) be the underlying group \(G\) equipped with the discrete topology. There is a canonical
surjective $*$-homomorphism
\[
q:C_X^*G\longrightarrow C_r^*(G_d), \quad q(\pi(g))=\lambda(g),
\]
and
\(
\ker q=\bigcap_{k\in G}\ker E_k.
\)
Thus the action representation is coefficient-separating if and only if $q$ is an isomorphism.
\end{proposition}

\begin{proof}
Let \(\tau=E_e\) and let \((\rho_\tau,\mathcal H_\tau,\Omega_\tau)\) be its GNS representation. The vectors \(\xi_g=\rho_\tau(\pi(g))\Omega_\tau,\) for $g \in G$, form an orthonormal basis of \(\mathcal H_\tau\). Indeed,
\[
\langle\xi_g,\xi_h\rangle
=\tau(\pi(h)^*\pi(g))
=\delta_{h,g},
\]
and their linear span is dense by cyclicity and the norm density of \(\operatorname{span}\pi(G)\) in \(C_X^*G\).

Let \(U:\mathcal H_\tau\to\ell^2(G_d)\) be the unitary determined by
\(U\xi_g=\delta_g\). Since\(  \rho_\tau(\pi(s))\xi_g=\xi_{sg},\) we have \(U\rho_\tau(\pi(s))U^*=\lambda(s)\), for $s\in G$. 
Thus \(q(T)=U\rho_\tau(T)U^*\) defines a \(*\)-homomorphism satisfying \(q(\pi(s))=\lambda(s)\). Its range is exactly \(C_r^*(G_d)\), by norm density and the closed-range property of \(C^*\)-algebra \(*\)-homomorphisms. Moreover, these identities uniquely determine \(q\), so \(q\) is canonical.

For finite sums, coefficient map gives \(E_k(T)=\langle q(T)\delta_e,\delta_k\rangle\). Both sides are continuous in \(T\), so this identity holds for every \(T\in C_X^*G\). Consequently,
\(E_k(T)=0\) for all \(k\in G\) if and only if \(q(T)\delta_e=0\).

Let \(R\) be the right regular representation, then we have \(R(s)\delta_t=\delta_{ts^{-1}}\). Every element of \(C_r^*(G_d)\) commutes with \(R(G)\).
Hence, if \(q(T)\delta_e=0\), then for every \(s\in G\),
\[
q(T)\delta_s
=q(T)R(s^{-1})\delta_e
=R(s^{-1})q(T)\delta_e
=0.
\]
It follows that \(q(T)=0\). The converse is immediate, and therefore
\[
\ker q=\bigcap_{k\in G}\ker E_k.
\]
Thus coefficient separation is equivalent to injectivity of \(q\). Since \(q\) is surjective, \(q\) is a \(*\)-isomorphism.
\end{proof}

\subsection{Applications on normalized trace}

Recall that a \emph{normalized trace} on a unital $C^{*}$-algebra $\mathcal A$ is a positive linear functional $\sigma:\mathcal A\to\mathbb C$ such that $\sigma(I)=1$ and $\sigma(ab)=\sigma(ba)$ for all $a,b\in\mathcal A$. Every positive unital functional has norm one.

Assume for the rest of this subsection that \(G\) has property \(\pa^X\). First, we show the existence of a normalized trace on $C_X^*G$. Let \(A=\operatorname{span}\{\pi(g):g\in G\}.\) By Lemma~\ref{lem:coefficient-functionals}, the identity coefficient map \(E_e: A\rightarrow\mathbb C\) is well-defined and bounded, with $|E_e(T)|\leq\|T\|, T\in A$. Therefore, \(E_e\) extends uniquely to a bounded linear functional \(\widetilde E_e:C_X^*G\rightarrow\mathbb C\). We still denoted by \(E_e\).

We claim that \(E_e\) is a normalized trace. Let \(T=\sum_{g\in F}a_g\pi(g)\in A.\) Then
\[
T^*T
=
\sum_{g,h\in F}\overline{a_g}a_h\pi(g^{-1}h),
\]
and hence
\[
E_e(T^*T)
=
\sum_{g\in F}|a_g|^2
\geq0.
\]
Thus \(E_e\) is positive on \(A\).

Let \(x\in C_X^*G\) be positive. Choose \(T_n\in A\) such that \(T_n\rightarrow x^{1/2}\) in norm. Then \(T_n^*T_n\rightarrow x\) in norm, and therefore
\[
E_e(x)
=
\lim_{n\to\infty}E_e(T_n^*T_n)
\geq0.
\]
Hence the extension of \(E_e\) is positive. Moreover, \(E_e(I)=E_e(\pi(e))=1.\) To verify the trace property, let \(S=\sum_{g\in F}a_g\pi(g)\) and \( T=\sum_{h\in F'}b_h\pi(h)\) belong to \(A\). Then
\[
E_e(ST)
=
\sum_{g\in F}a_gb_{g^{-1}}
=
\sum_{h\in F'}b_ha_{h^{-1}}
=
E_e(TS),
\]
where missing coefficients are understood to be zero. By density of \(A\) and continuity of \(E_e\), it follows that \(E_e(xy)=E_e(yx)\) for all \(x,y\in C_X^*G\). Thus \(E_e\) is a normalized trace.

The following shows the uniqueness of the normalized trace.

\begin{proposition}\label{Trace}
If a group $G$ has the property \(\pa^{X}\), then the action-based $C^*$-algebra $C^*_{X}G$ admits a unique normalized trace.
\end{proposition}

\begin{proof}

Suppose $\sigma$ is any normalized trace on $C^*_{X}G$. By assumption, $G$ has the property \(\pa^{X}\). Given any $h \in G \setminus \{e\}$, there exist an element $g \in G$ and a constant $C$ such that for all $a \in \ell^{2}(\mathbb{Z}^{+})$, we have 
\[
\left\| \sum_{j=1}^{\infty} a_j \pi(g^{-j} h g^j) \right\|_{\mathrm{op}} \leq C \left\| a \right\|_{\ell^{2}}.
\]

Since $\sigma(ab)=\sigma(ba)$ for all $a, b \in C^*_{X}(G)$ and $\pi(g^{-j})\pi(g^{j})=I$, we always have $\sigma(\pi(h))=\sigma(\pi(g^{-j}hg^{j}))$ for all $j\in \mathbb{Z}^{+}$. Then
\[
|\sigma(\pi(h))|
=\left|\sigma\left(\frac{1}{J}\sum_{j=1}^{J}
\pi(g^{-j}hg^{j})\right)\right|.
\]
Take 
$$a_{j}=\begin{cases}
    \frac{1}{J}, & \mbox{if}~j \leq J,\\
    0, & \mbox{if}~j>J.
\end{cases}$$
By the property \(\pa^{X}\), we have 
\[
|\sigma(\pi(h))|
\leq \|\sigma\|
\left\|\frac{1}{J}\sum_{j=1}^{J}\pi(g^{-j}hg^{j})\right\|
\leq\frac{C}{\sqrt J},
\]
because \(\|\sigma\|=1\). Hence $\sigma(\pi(h))=0$ as
$J\to\infty$. It follows that \(\sigma(\sum_h a_h\pi(h))=a_e\), so \(\sigma=E_e\) on the dense subalgebra \(A=\pi(\mathbb CG)\). Continuity gives \(\sigma=E_e\) on \(C_X^*G\).
\end{proof}

\section{Applications on big mapping class groups}\label{5}

Throughout this section, \(S\) denotes a connected, orientable, boundaryless surface of infinite type with positive complexity. We assume that \(S\) contains a connected nondisplaceable finite type subsurface. All subsurfaces considered here are closed in \(S\), with boundary consisting of finitely many pairwise disjoint simple closed curves. Finite type subsurfaces may have punctures and need not be compact. We call a finite type subsurface \emph{clean} if its boundary curves are essential in \(S\) and every complementary component has infinite type. Let \(\widehat T\) be obtained from \(T\) by replacing its boundary components with punctures, retaining all original punctures. Write the complexity \(\xi(\widehat T)=3g(T)-3+p(T)+|\partial T|\), where \(p(T)\) is the original punctures.

\begin{lemma}\label{lem:clean-enlargement}
Let \(B\subset S\) be a finite union of finite type subsurfaces and compact subsets. Let \(\mathcal D\) be a finite collection of essential simple closed curves in \(S\) and let \(N\ge2\) be an integer. There exists a connected clean finite type subsurface \(T\) such that
\[
B\cup\bigcup_{c\in\mathcal D}c\subset\operatorname{int}(T),
\quad
\xi(\widehat T)\ge N,
\]
and every \(c\in\mathcal D\) is essential in \(\widehat T\).
\end{lemma}

\begin{proof}
Choose an \(N\)-component multicurve \(\mathcal E=\{a_{1}, \ldots, a_{N}\}\) in \(S\) and,
for each \(c\in\mathcal D\cup\mathcal E\), an essential curve \(d_c\)
with \(i_S(c,d_c)>0\).
Take a
connected finite type surface \(U\) such that \[B \cup \bigcup_{c \in \mathcal{D} \cup \mathcal{E}} c \cup \bigcup_{c} d_{c} \subset \operatorname{int}(U).\]
Adjoining the closures of all finite type components of \(S\setminus U\), the resulting surface, denoted by $T$, is clean.

The choice of $d_{c}$ ensures that every curve in
\(\mathcal D\cup\mathcal E\) is essential in \(\widehat T\).
Since \(T\hookrightarrow\widehat T\) is a homotopy equivalence,
the components of \(\mathcal E\) remain pairwise nonisotopic.
Thus \(\mathcal E\) is an \(N\)-component multicurve in
\(\widehat T\), giving \(\xi(\widehat T)\ge N\).
\end{proof}

Apply Lemma~\ref{lem:clean-enlargement} to fix a connected clean nondisplaceable subsurface \(K\subset S\) with \(\xi(\widehat K)\ge3\). Notice that if a connected subsurface $K^\prime$ contains a nondisplaceable subsurface $K$, then $K^{\prime}$ is also nondisplaceable. 

For a clean nondisplaceable subsurface \(T\) with \(\xi(\widehat T)\ge3\), a finite curve system filling \(\widehat T\) determines \(T\) up to isotopy in \(S\). In particular, take a regular neighborhood of representatives in minimal position and adjoin its disk and once-punctured-disk complementary components in \(S\). We call \(\widehat T\) its filled support.

\subsection{Quasi-trees of curve graphs for big mapping class groups}Suppose that $S$ contains a nondisplaceable subsurface $K$.
The curve graph $\mathcal C(S)$ has vertices the isotopy classes of essential simple closed curves on $S$, and an edge joins two vertices if they admit disjoint representatives. Let $\mathcal{C}_{S}(K)$ be the graph whose vertices are the homotopy classes of simple closed curves on $S$ that have a representative contained in $K$ which is essential in $K$, where an edge joins two distinct isotopy classes if they can be realized by disjoint representatives in $S$. The curve graph $\mathcal{C}_{S}(K)$ is an induced subgraph of the curve graph $\mathcal{C}(S)$ of $S$. Also, $\mathcal{C}_{S}(K)$ only depends on the isotopy class of $K$: if $K$ and $K_{1}$ are isotopic, then there is a natural identification between $\mathcal{C}_{S}(K)$ and $\mathcal{C}_{S}(K_{1})$. We write $[K]$ for the isotopy class of $K$. Then we consider $\mathbf{Y}_{K}$ to be the orbit of $K$ under the metric-preserving action of $\operatorname{Map}(S)$. Each $K$ has an associated geodesic metric space $\mathcal{C}_{S}(K)$.

As \(K\) is nondisplaceable and clean, given any two non-isotopic subsurfaces \(K_{1}, K_{2} \in (\operatorname{Homeo}(S) \cdot K)\), at least one of the boundary components of \(K_{2}\) intersects the subsurface
\(K_{1}\) in an essential curve or arc (since $K$ is nondisplaceable we have $K_1\cap K_2\neq\varnothing$, and if
$\partial K_2\cap K_1=\varnothing$ then $K_1$ and $K_2$ are nested and isotopic, contradicting $K_1\neq K_2$). For each of these arcs, closing them up along $\partial{K_{1}}$ in potentially two different ways gives us a collection of essential simple closed curves in $K_{1}$. For \(K_{1}, K_{2} \in (\operatorname{Map}(S) \cdot K)\), we define a
projection
$$
\pi_{K_{1}}\left(K_{2}\right) \subseteq \mathcal{C}_{S}\left(K_{1}\right),
$$
to be the union of all essential simple closed curves and closed up arcs coming from $\partial K_2\cap K_1$. Then we define distances between subsurface projections as follows. For \(K_{1}, K_{2}, K_{3} \in (\operatorname{Map}(S) \cdot K)\)
, define
$$
d_{K_{1}}(K_{2}, K_{3})=\operatorname{diam}_{\mathcal{C}_{S}(K_{1})}\left(\pi_{K_{1}}(K_{2}) \cup \pi_{K_{1}}(K_{3})\right).
$$

By \cite{BBFS}, the above projections can be modified equivariantly so that the projection data satisfy the strong projection axioms. Avoiding abuse of notation, we also denote the modified projections by \(\pi'\) and write \(\theta\) for the corresponding projection constant.

Following projection data above, we can define a connected graph $\mathbb P:=\mathcal{P}_{\alpha}(\mathbf{Y}_{K})$, called \emph{projection complex}, for $\alpha \geq 2 \theta$ which has the set of vertices as $\mathbf{Y}_{K}$, and there is an edge between $K_1, K_2 \in \mathbf{Y}_{K}$ if and only if $d_{K_3}(K_1, K_{2}) \leq \alpha$ for every $K_{3} \in \mathbf{Y}_{K}$. Then they build the space $\mathbb{X}:=\mathcal{C}_{\alpha}(\mathbf{Y}_{K})$ by replacing each vertex $K_{i} \in \mathbf{Y}_{K}$ with $\mathcal{C}(K_{i})$ and joining points in $\pi_{K_{i}}(K_{j})$ with points $\pi_{K_{j}}(K_{i})$ by an edge of length $L(\alpha) >0$ whenever $K_{j}$ and $K_{i}$ has an edge between them in $\mathcal{P}_{\alpha}(\mathbf{Y}_{K})$, i.e. $d_{\mathbb P}(K_{j}, K_{i})=1$ for $i \neq j$. 

The curve graphs \(\mathcal C_S(K_1)\), for \(K_1\in\mathbf Y_K\), are uniformly hyperbolic because they are all isometric to \(\mathcal C_S(K)\). Therefore, \cite[Corollary~6.8]{BBFS} implies that \( \mathbb X\) is hyperbolic. In fact, the space \(\mathbb X\) also carries a continuous, nonelementary isometric action of \(\operatorname{Map}(S)\) by \cite[Theorem~2.9]{HQR20}.

Then we keep the Koopman notion for the mapping class group \(\MapS\) acting by isometries on \(\mathbb{X}\). First, notice that \(\mathbb{X}^{(0)}\), the set of vertices of \(\mathbb{X}\), is countable. Since \(S\) is second countable, the set of isotopy classes of finite type subsurfaces of \(S\) is countable. Hence the set $\{\mathcal{C}_{S}(K_{1})| K_{1}\in \mathbf{Y}_{K}\}$ is countable. Moreover, for each \(K_{1}\in \mathbf{Y}_{K}\), the curve graph \(\mathcal{C}_{S}(K_{1})\) has countably many vertices. Therefore the vertex set of \(\mathbb{X}\) is countable.

Let \(|\mathbb X|\) denote the path-metric realization of \(\mathbb X\), and let \(V_{\mathbb X}=\mathbb X^{(0)}\), equipped with the restriction of the path metric on \(|\mathbb X|\). The notions below (geodesics, quasi-axes, Gromov boundaries, and loxodromic or elliptic dynamics) are made in the path-metric realization \(|\mathbb X|\), whereas the action representation is defined on the countable set \(V_{\mathbb X}\).  The inclusion \(V_{\mathbb X}\hookrightarrow |\mathbb X|\) is an equivariant quasi-isometry: all edges have one of the fixed lengths used in the Bestvina-Bromberg-Fujiwara-Sisto \cite{BBFS} construction, so the vertex set is uniformly coarsely dense. Consequently, an element is elliptic or loxodromic on the vertex set if and only if it is so on \(|\mathbb X|\).  We henceforth use the abbreviation \(\ell^2(\mathbb X):=\ell^2(V_{\mathbb X})\).
Thus
\[
 \ell^2(\mathbb X)
 =\Big\{\xi:V_{\mathbb X}\to\mathbb C:
   \sum_{x\in V_{\mathbb X}}|\xi(x)|^2<\infty\Big\}.
\]
We write \(\mathcal U(\ell^2(\mathbb X))\) and \(\mathcal B(\ell^2(\mathbb X))\) for the unitary and bounded operators, respectively. The vertex action preserves the counting measure and induces \(\pi:\MapS\longrightarrow\mathcal U(\ell^2(\mathbb X)),(\pi(g)\xi)(x)=\xi(g^{-1}x), \forall x\in V_{\mathbb X}\).

\begin{lemma}\label{a}
Let \(S\) be a connected orientable boundaryless infinite type surface with positive complexity, and assume that \(S\) contains a nondisplaceable connected subsurface of finite type. Then the action of \(\MapS\) on $\mathbb{X}$ is faithful: for any $g \in \MapS$, if $g \cdot x=x$ for all $x \in \mathbb{X}$, then $g=e$.
\end{lemma}

\begin{proof}
Let \(K\subset S\) be the fixed nondisplaceable subsurface above and let \(\mathbb X\) be the corresponding fixed quasi-tree of metric spaces.

Assume that \(g\in\MapS\) fixes every vertex of \(\mathbb X\). By the construction of the quasi-tree of metric spaces, every \(\mathcal C_S(K')\), with \(K'\in\MapS\cdot K\), embeds \(\MapS\)-equivariantly as a subspace of \(\mathbb X\). Hence \(g\) fixes every vertex of \(\mathcal C_S(K')\).

By \cite[Lemma~2.5]{HQR20}, the pointwise stabilizer of \(\mathcal C_S(K')\) is equal to \(\operatorname{Fix}_{\MapS}(K')\). Therefore, \(g\in \operatorname{Fix}_{\MapS}(K')\) for every $K'\in\MapS\cdot K$. In particular, \(g\) preserves the isotopy class of every boundary component of every subsurface \(K'\).

Claim that $g$ fixes all essential simple closed curves $c$ in $S$. Let $c$ be an essential simple closed curve in $S$ and let \(b_0\) be an essential boundary component of \(K\). Such a component exists and is essential because $K$ is a proper clean finite type subsurface of the infinite type surface $S$. Choose a connected finite type subsurface \(L\subset S\) of positive complexity whose interior contains both \(b_0\) and \(c\), with both curves essential in \(L\). Let \(\varphi\in \operatorname{Map}(L)\) be pseudo-Anosov, extended by the identity to \(S\). Since the stable lamination of \(\varphi\) fills \(L\), we have \(i(\varphi^n(b_0),c)>0\) for all sufficiently large \(n\). Taking \(f=\varphi^n,K_1=f(K)\) and \(b=f(b_0)\), we obtain \(b\subset\partial K_1\) with the number of intersection \(i(b,c)>0\).

For every \(N\in\mathbb Z\), make Dehn twists along $c$ for $N$ times and define
\(K_N:=T_c^N(K_1), b_N:=T_c^N(b).\)
Then \(K_N\in(\MapS \cdot K)\) and \(b_N\subset\partial K_N\). So \(g\) fixes every boundary component of \(K_N\), hence \(g b_N=b_N\) for all \(N\). On the other hand,
\[
g(b_N)
=
gT_c^N(b)
=
T_{gc}^N(gb)
=
T_{gc}^N(b).
\]
Since \(g\) fixes the isotopy class of \(b_N=T_c^N(b)\), it follows that
\(T_{gc}^N(b)=T_c^N(b)\) for every \(N\in\mathbb Z\).

We claim that this implies \(gc=c\). Suppose, towards a contradiction, that \(gc\neq c\). Choose a connected finite type subsurface $R\subset S$ of sufficiently large complexity whose interior contains $c\cup gc$. By the change-of-coordinates argument used in the proof of \cite[Fact~3.6]{FM11}, applied inside $R$ (after enlarging $R$ if necessary), there exists an essential simple closed curve \(d\) such that
\(i(gc,d)=0\) and \(i(c,d)>0.\)
Since \(T_{gc}\) fixes the isotopy class of \(d\), we have \(i(T_{gc}^N(b),d)=i(b,d)\) for every \(N\).

On the other hand, by \cite[Proposition~3.4]{FM11},
\[
\left|
i(T_c^N(b),d)
-
|N|\,i(c,b)i(c,d)
\right|
\leq i(b,d).
\]
Since \(i(c,b)>0\) and \(i(c,d)>0\), the quantity
\(i(T_c^N(b),d)\) tends to infinity as \(|N|\to\infty\). This
contradicts
\(
T_{gc}^N(b)=T_c^N(b).
\)
Therefore \(gc=c\).

Thus $g$ fixes the isotopy class of every essential simple closed curve.
For the boundaryless infinite type surfaces considered here, the
faithfulness theorem of Hern{\'a}ndez-Morales-Valdez\cite{HMV19} now implies that
$g=e$ in $\operatorname{Map}(S)$.
\end{proof}


\subsection{The property \texorpdfstring{$\pn^{\mathbb{X}}$}{P {naive} X} for big mapping class groups}

 In \cite{TL24}, Lou proved that $\MapS$ has the property $\pn$ when $S$ contains a nondisplaceable connected subsurface of finite type using the action of $\MapS$ on $\mathbb{X}$. Given a finite collection $F=\{h_1,\dots,h_n\}\subset \operatorname{Map}(S)\setminus \{e\}$, she tailors the space $\mathbb{X}$ to capture the dynamics of those mapping classes. This construction of $\mathbb{X}$ depends on the chosen subset $F$. In fact, for each new finite set of mapping classes, she gets a new hyperbolic $\mathbb{X}$ adapted to them. This flexibility is crucial for applying a ping-pong argument in a naive setting.

We next strengthen the above by establishing $\pn^{\mathbb{X}}$ using an action of $\MapS$ on a fixed hyperbolic space $\mathbb{X}$, rather than an $F$-dependent family of spaces $\mathbb{X}_{F}$.

\begin{theorem}\label{TL1}
Let \(S\) be a connected orientable infinite type surface without boundary containing a connected nondisplaceable finite type subsurface. Then there exists a fixed hyperbolic \(\mathbb X\) such that \(\MapS\) has property \(\pn^{\mathbb X}\).
\end{theorem}

Firstly, we can choose a sufficiently large clean nondisplaceable subsurface \(K\) by Lemma~\ref{lem:clean-enlargement} and then fix it. Using the surface projections on \(\mathbf Y_K=\MapS\cdot K\), we can get the hyperbolic space \(\mathbb X\). Then for any finite set \(F\subset\MapS\setminus\{1\}\), we find a WWPD loxodromic element \(g\) on the fixed space \(\mathbb X\) such that no nonidentity power of any \(h\in F\) sends either endpoint of \(g\) to either endpoint of \(g\). Then we use the WWPD property of \(g\) and the separation condition to show that for each \(h\in F\), a sufficiently large power \(g^Q\) and \(h\) generate \(\langle h\rangle*\langle g^Q\rangle\) and every elliptic subgroup is conjugate to \(\langle h\rangle\). Since \(F\) is finite, one exponent \(Q\) works for all its elements.

For a clean finite type subsurface \(M\), retain all its original punctures when forming \(\widehat M\) as above, and write
\[
\rho_M:\operatorname{Stab}_{\MapS}(M)
\longrightarrow\operatorname{Map}(\widehat M),
\operatorname{Fix}_{\MapS}(M)=\ker\rho_M.
\]
We write \(\bar u=\rho_M(u)\) and use the restriction homomorphisms of \cite[Lemmas~2.4--2.5]{HQR20}. For a loxodromic element \(u\), \(E(u)\) denotes the stabilizer of its unordered pair of boundary fixed points.

Next, we construct a finitely supported WWPD loxodromic element whose endpoint stabilizer is controlled by its action on the support.

\begin{proposition}
\label{prop:TL1-support}
Every clean finite type subsurface \(M_0\subset S\) is contained in the interior of a connected clean finite type subsurface \(M\) for which \(\operatorname{Map}(\widehat M)\) has no nontrivial finite normal subgroup and there is an element \(f\), supported on \(M\), with the following properties:
\begin{enumerate}
 \item \(\bar f\) is pseudo-Anosov and \(f\) is WWPD loxodromic  on the fixed space \(\mathbb X\);
 \item if \(u\in\MapS\) sends an endpoint of \(f\) to an endpoint of \(f\), then \(u\in\operatorname{Stab}_{\MapS}(M)\) and \(\bar u\in E(\bar f)\).
\end{enumerate}
Moreover, those conclusions hold for every conjugate of \(f\) by an element of \(\operatorname{Stab}_{\MapS}(M)\).
\end{proposition}

\begin{proof}
By Lemma~\ref{lem:clean-enlargement}, we can choose a connected clean finite type subsurface \(L\) whose interior contains \(K\), with every component of \(\partial K\) essential in \(\widehat L\). Then we choose a mapping class \(q\) supported on \(L\) whose capped restriction is pseudo-Anosov. The subsurfaces \(q^jK\) are pairwise distinct. We take two intersecting essential curves \(a,a'\) in an infinite type component of \(S\setminus L\). Applying Lemma~\ref{lem:clean-enlargement} with
\(
B=M_0\cup L,
\mathcal D=\partial L\cup\{a,a'\},
N=3,
\)
we obtain a connected clean finite type subsurface \(M\) such that
\(
M_0\cup L\cup a\cup a'
\subset\operatorname{int}(M),\)
\(\xi(\widehat M)\ge3,
\)
and every curve in \(\partial L\cup\{a,a'\}\) is essential in
\(\widehat M\). The natural action of \(\operatorname{Map}(\widehat M)\) on \(\mathcal C(\widehat M)\) is faithful. Let \(N\) be a finite normal subgroup of \(\operatorname{Map}(\widehat M)\). For every essential curve \(c\), some positive power of the Dehn twist \(T_c\) centralizes \(N\). Hence every \(n\in N\) fixes \(c\). Since \(c\) was arbitrary, \(n\) fixes every vertex of \(\mathcal C(\widehat M)\). Faithfulness of the action gives \(n=1\). Therefore, \(N=\{1\}\).

Choose \(t\in\operatorname{Map}(M,\partial M)\) with \(d_{\mathcal C(\widehat M)}(a,\bar t a)>4\). By \cite[Theorem~1.3 and Remark~1.4]{Wat20}, there are only finitely many exceptional subsurfaces disjoint from \(a\) for this choice of \(t\). Choose \(Y=q^jK\) outside that finite family and a mapping class \(\psi\) supported on \(Y\) whose capped restriction to \(\widehat Y\) is pseudo-Anosov. Extend \(\psi\) by the identity outside \(Y\). Its capped restriction is pure, fixes \(a\), and has \(Y\) as its loxodromic domain, so Theorem~1.3 of \cite{Wat20} gives
\(f_N=\psi^Nt\) for all sufficiently large \(N\).

First, \(Y\ne tY\). Otherwise, both \(a\) and \(\bar t a\)
would be disjoint from \(\partial Y\). This would give
\(d_{\mathcal C(\widehat M)}(a,\bar t a)\leq2,\)
contrary to the choice of \(t\). Hence \(t^{-1}Y\ne Y\) and
\(tY\ne Y\).

The sets \(\pi'_Y(t^{-1}Y)\) and \(\pi'_Y(tY)\) are nonempty bounded subsets of
\(\mathcal C_S(Y)\). Since \(\psi Y=Y\), equivariance of the
projection data gives
\[
\pi'_Y(\psi^N tY)=\psi^N \pi'_Y(tY).
\]
Moreover, \(\psi\) acts loxodromically on \(\mathcal C_S(Y)\).
Therefore,
\[
\operatorname{diam}_{\mathcal C_S(Y)}
   \bigl(\pi'_Y(t^{-1}Y)\cup\psi^N \pi'_Y(tY)\bigr)\longrightarrow\infty.
\]
Since \(f_N=\psi^Nt\), we have
\(f_N^{-1}Y=t^{-1}Y\) and \(f_NY=\psi^NtY.\)
Consequently,
\[
\Delta_N
:=d_Y(f_N^{-1}Y,f_NY)
=d_Y(t^{-1}Y,\psi^NtY)
\longrightarrow\infty.
\]

For \(\Delta_N>\alpha\), the order and concatenation properties of standard paths \cite[Proposition~2.3, Lemma~2.4 and Lemma~3.1]{BBFS} imply that
\[
 A_N:=\bigcup_{j\in\mathbb Z}f_N^j[Y,f_NY]_{\mathrm{std}}
\]
is a bi-infinite standard path in \(\mathbb P\) and every finite subpath is the standard path between its endpoints. Standard paths are uniform quasi-geodesics \cite[Corollary~3.8]{BBFS}, so \(f_N\) is loxodromic on \(\mathbb P\), and hence on \(\mathbb X\) by the equivariant coarse Lipschitz collapse \(\mathbb X\to\mathbb P\). For an interior vertex \(P\) of a standard path \(A\), write \(\angle_A(P)=d_P(P_-,P_+)\), where \(P_-,P_+\) are its adjacent vertices. At \(f_N^jY\), this angle is \(\Delta_N\). Since 
\([Y,f_NY]_{\mathrm{std}}=\psi^N[Y,tY]_{\mathrm{std}}\), all other angles are bounded by a constant \(B\geq\alpha\) independent of \(N\). Fix \(N\) large enough so that \(\bar f_N\) is pseudo-Anosov and \(\Delta_N>B\). From now on, write \(f=f_N\).

Figure~\ref{fig:TL1-support} shows how these projection distances distinguish the vertices \(f^jY\).
\begin{figure}[htbp]
\centering
\begin{tikzpicture}[x=1cm,y=1cm,font=\small,
  ordinary/.style={circle,draw=black!70,fill=white,inner sep=1.3pt},
  marked/.style={circle,draw=blue!45!black,fill=blue!45!black,inner sep=2.1pt}]
  \node[anchor=west] at (0,3.65) {(a) Projection distances in $\mathcal C_S(Y)$};
  \draw[draw=black!65,fill=black!6] (2,2.55) ellipse (.72 and .27);
  \draw[draw=black!65,fill=black!6] (9.6,2.55) ellipse (.72 and .27);
  \foreach \p in {1.7,1.95,2.25,9.3,9.55,9.85}
    \fill[black!65] (\p,2.55) circle (.022);
  \node at (2,2.03) {$\pi_Y(t^{-1}Y)$};
  \node at (9.6,2.03) {$\pi_Y(\psi^NtY)$};
  \draw[black!65] (1.28,2.96)--(1.28,3.10)--(10.32,3.10)--(10.32,2.96);
  \node[fill=white,inner sep=3pt] at (5.8,3.10)
    {$\Delta_N=\operatorname{diam}\bigl(\pi_Y(t^{-1}Y)\cup
       \pi_Y(\psi^NtY)\bigr)\longrightarrow\infty$};
  \node[anchor=west] at (0,1.30) {(b) The standard path $A_N\subset\mathbb P$, with $f=f_N$};
  \draw[densely dashed,black!60] (0,.6)--(.6,.6);
  \draw[black!70] (.6,.6)--(11.4,.6);
  \draw[densely dashed,black!60] (11.4,.6)--(12,.6);
  \foreach \p in {1.8,2.5,3.2,4.6,5.3,6,7.4,8.1,8.8,10.2}
    \node[ordinary] at (\p,.6) {};
  \node[marked,label=below:{$f^{-1}Y$}] at (1.1,.6) {};
  \node[marked,label=below:{$Y$}] at (3.9,.6) {};
  \node[marked,label=below:{$fY$}] at (6.7,.6) {};
  \node[marked,label=below:{$f^2Y$}] at (9.5,.6) {};
  \draw[->,>=stealth,black!70] (3.9,-.50).. controls (4.6,-.90) and (6,-.90)..(6.7,-.50);
  \node[fill=white,inner sep=2pt] at (5.3,-.82) {$f$};
  \node[marked] at (.30,-1.38) {};
  \node[anchor=west] at (.53,-1.38) {$\angle_{A_N}(f^jY)=\Delta_N>B$};
  \node[ordinary] at (7.25,-1.38) {};
  \node[anchor=west] at (7.48,-1.38) {other projection angles $\leq B$};
\end{tikzpicture}
\caption{Large projections distinguish the translates of $Y$. The union of the two bounded projection sets has diameter $\Delta_N\to\infty$. On $A_N$, the vertices $f^jY$ are exactly those with projection angle greater than $B$. Any element preserving $A_N$ therefore permutes these vertices, whose boundary curves collectively fill $M$. The drawing represents projection distances and standard-path order, not Euclidean angles.}
\label{fig:TL1-support}
\end{figure}

We first verify WWPD. Choose an essential component \(b\) of \(\partial Y\). Since \(\bar f\) is pseudo-Anosov, there is \(r>0\) such that \(b\) and \(f^rb\) fill \(\widehat M\). Thus the curve system
\(
\partial Y\cup\partial(f^rY)
\)
fills \(\widehat M\). Its filled support in \(S\) is \(M\). Let
\[
H=\operatorname{Stab}(Y)\cap\operatorname{Stab}(f^rY).
\]
Every element of \(H\) preserves this curve system and hence preserves \(M\). Therefore,
\[H\leq\operatorname{Stab}_{\MapS}(M).\]
The image \(\rho_M(H)\) is finite because the stabilizer of a finite filling curve system on \(\widehat M\) is finite by the Alexander method. Every element of \(\operatorname{Fix}_{\MapS}(M)\) commutes with \(f\), since it is represented on \(M\) by boundary twists. Consequently
\[
 [H:H\cap\operatorname{Stab}_{\MapS}(f^-,f^+)]<\infty,
\]
where the endpoint pair is ordered. 

Fix \(x\in\mathcal C_S(Y)\) and let \(R>0\). For each
\(\ell>0\), set
\[
\mathcal Z_{R,\ell}
=\left\{
u\in\MapS :
d_{\mathbb X}(x,ux)\leq R,\ 
d_{\mathbb X}(f^\ell x,uf^\ell x)\leq R
\right\}.
\]
Under the equivariant collapse
\(\mathbb X\to\mathbb P\), the points \(x\) and \(f^\ell x\)
map to \(Y\) and \(f^\ell Y\), respectively. Hence there is a
constant \(D_R\) such that every \(u\in\mathcal Z_{R,\ell}\)
satisfies
\[
d_{\mathbb P}(Y,uY)\leq D_R,
d_{\mathbb P}(f^\ell Y,uf^\ell Y)\leq D_R.
\]

For sufficiently large \(\ell\), the standard-path quadrilateral
argument from the proof of \cite[Theorem~3.9]{BBFS} provides a long middle segment \(I\subset A_N\). The action of any \(u\in\mathcal Z_{R,\ell}\) on \(I\) is one of finitely many possible translations. Choose \(I\) long enough to contain
\(f^kY\) and \(f^{k+r}Y\) for some \(k\in\mathbb Z\).

Suppose that \(u,v\in\mathcal Z_{R,\ell}\) induce the same map on \(I\). Then \(v^{-1}u\) fixes both \(f^kY\) and \(f^{k+r}Y\).
Therefore,
\[
v^{-1}u
\in
\operatorname{Stab}(f^kY)\cap
\operatorname{Stab}(f^{k+r}Y)
=
f^kHf^{-k}.
\]
The previously established bound
\[
[H:H\cap \operatorname{Stab}_{\MapS}(f^-,f^+)]<\infty
\]
implies
\[
[f^kHf^{-k}:f^kHf^{-k}\cap \operatorname{Stab}_{\MapS}(f^-,f^+)]<\infty.
\]
Thus, for each possible action on \(I\), the corresponding elements of \(\mathcal Z_{R,\ell}\) lie in finitely many left cosets of \(E\). There are only finitely many possible actions on \(I\). Hence \(\mathcal Z_{R,\ell}\) meets only finitely many left cosets of \(E\).
This is WWPD(1) in \cite[Proposition and Definition~2.3]{HM}.

We now prove the endpoint assertion. Suppose that \(u\in\MapS\) sends one endpoint of \(f\) to an endpoint of \(f\). Then
\[
u\{f^-,f^+\}\cap\{f^-,f^+\}\ne\varnothing.
\]
If \(u\) did not preserve \(\{f^-,f^+\}\), then \cite[Proposition and Definition~2.3(6)]{HM} would imply that \(f\) and \(ufu^{-1}\) have disjoint endpoint sets. Therefore,
\(
u\{f^-,f^+\}=\{f^-,f^+\}.
\)

Let \(\mathcal Q\) be an orbit quasi-axis of \(f\) in \(\mathbb X\).
The path \(u\mathcal Q\) is an orbit quasi-axis of \(ufu^{-1}\).
These two quasi-axes have the same endpoint pair, so they are at
finite Hausdorff distance. The collapse map sends \(f^jx\) to \(f^jY\) for every \(j\in\mathbb Z\). The vertex set
\(
\{f^jY:j\in\mathbb Z\}
\)
is coarsely dense in \(A_N\). Similarly, the collapse map sends \(uf^jx\) to \(uf^jY\), and these vertices are coarsely dense in \(uA_N\). Since the collapse map is coarsely Lipschitz and the two orbit quasi-axes are at finite Hausdorff distance in \(\mathbb X\), the paths \(A_N\) and \(uA_N\) are at finite Hausdorff distance in \(\mathbb P\).

We claim that two bi-infinite standard paths at finite Hausdorff distance are equal. Let \(P\) be a vertex of the first path. Choose vertices \(X\) and \(Z\) on opposite sides of \(P\), sufficiently far from \(P\). Choose vertices \(X'\) and \(Z'\) on the second path that are uniformly close to \(X\) and \(Z\), respectively. The standard-path quadrilateral argument in \cite[Lemma~3.6 and the proof of Theorem~3.9]{BBFS} places \(P\) on the standard path from \(X'\) to \(Z'\). This path is a subpath of the second bi-infinite path. Thus the first path is contained in the second. The reverse inclusion follows by symmetry. Therefore, \(uA_N=A_N.\)

The vertices \(f^jY\) are precisely the vertices of \(A_N\) whose projection angles are greater than \(B\). Equivariance preserves projection angles. Hence \(u\) permutes the set
\(
\{f^jY:j\in\mathbb Z\}.
\)
The order on \(A_N\) then gives constants
\(\varepsilon\in\{1,-1\}\) and \(k\in\mathbb Z\) such that
\(
uf^jY=f^{\varepsilon j+k}Y
\) for \(j\in\mathbb Z.
\)

The curve system
\(
\partial Y\cup\partial(f^rY)
\)
has filled support \(M\). Its image under \(u\) is
\(
\partial(f^kY)\cup
\partial(f^{\varepsilon r+k}Y).
\)
The latter curve system also has filled support \(M\). On the other
hand, the filled support of the image is \(u(M)\). Therefore,
\(u(M)=M.\)
In particular, \(u\in\operatorname{Stab}_{\MapS}(M)\), so
\(\bar u=\rho_M(u)\) is defined.

Let \(\lambda^+\) and \(\lambda^-\) be the stable and unstable
projective laminations of \(\bar f\). Fix a component
\(b\subset\partial Y\). For every \(j\in\mathbb Z\), there is a
component \(c_j\subset\partial Y\) such that
\(
uf^jb=f^{\varepsilon j+k}c_j.
\)
Since \(\partial Y\) has finitely many components, there is a sequence \(j_n\to+\infty\) for which \(c_{j_n}=c_+\) is constant. There is also a sequence \(j_n'\to-\infty\) for which \(c_{j_n'}=c_-\) is constant. Taking projective limits gives
\[
\bar u(\lambda^+)=
\begin{cases}
\lambda^+,&\varepsilon=1,\\
\lambda^-,&\varepsilon=-1,
\end{cases}
\qquad
\bar u(\lambda^-)=
\begin{cases}
\lambda^-,&\varepsilon=1,\\
\lambda^+,&\varepsilon=-1.
\end{cases}
\]
Thus \(\bar u\) preserves the unordered pair
\(\{\lambda^-,\lambda^+\}\). Hence
\(
\bar u\in E(\bar f).
\)

Finally, let \(h\in\operatorname{Stab}_{\MapS}(M)\) and \(f_h=hfh^{-1}\). Conjugation preserves the loxodromic and WWPD properties, and \(f_h^\pm=hf^\pm\). Moreover,
\( \overline{f_h}=\bar h\bar f\bar h^{-1},\) so \(\overline{f_h}\) is pseudo-Anosov. If \(u\) sends an endpoint of \(f_h\) to an endpoint of \(f_h\), then \(h^{-1}uh\) sends an endpoint of \(f\) to an endpoint of \(f\). Hence
\(h^{-1}uh\in\operatorname{Stab}_{\MapS}(M) \quad\text{and}\quad \overline{h^{-1}uh}\in E(\bar f).\)
Consequently,
\(u\in\operatorname{Stab}_{\MapS}(M) \quad\text{and}\quad \bar u\in E(\overline{f_h}).\)
Thus all the stated conclusions hold for \(f_h\).
\end{proof}

The above proposition now allows us to choose a single WWPD loxodromic element whose endpoint pair is simultaneously separated from all its translates by the nontrivial powers of elements of $F$.

\begin{proposition}
\label{prop:TL1-separation}
For every finite subset \(F\subset\MapS\setminus\{1\}\), there exists \(g\in\MapS\) which is WWPD loxodromic for the fixed action of \(\MapS\) on \(\mathbb X\) and satisfies
\begin{equation}
\label{eq:TL1-separation}
h^m\{g^-,g^+\}\cap\{g^-,g^+\}=\varnothing
\qquad
(h\in F,\ m\in\mathbb Z,\ h^m\ne1).
\end{equation}
\end{proposition}

\begin{proof}
We first choose a clean finite type subsurface \(M_0\) such that \(\langle h\rangle\cap\operatorname{Fix}(M_0)=\{1\}\) for all \(h\in F\). For an infinite order \(h\), some essential curve has infinite \(h\)-orbit, by \cite[proof of Lemma~12]{ACCL21}. Choose one such curve for each infinite order \(h\in F\). For each finite order \(h\in F\) and each nonidentity power \(h^m\), choose an essential curve moved by \(h^m\). Such a curve exists because the natural action of \(\operatorname{Map}(S)\) on the isotopy classes of essential simple closed curves in \(S\) is faithful \cite[Section~4]{ACCL21}.
By Lemma~\ref{lem:clean-enlargement}, we can choose a connected clean finite type subsurface \(M_0\) with \(\xi(\widehat M_0)\ge3\), containing all these curves in its interior, with the curves essential in \(\widehat M_0\). Therefore, the subsurface \(M_0\) satisfies
\[
\langle h\rangle\cap\operatorname{Fix}(M_0)=\{1\},
\] for all \(h\in F\). Indeed, every element of \(\operatorname{Fix}(M_0)\) fixes all the chosen curves, whereas every nonidentity element of \(\langle h\rangle\) moves at least one of them.

Apply Proposition~\ref{prop:TL1-support} to obtain \(M\) and \(f\). Since \(M_0\subset\operatorname{int}(M)\), we have
\[\operatorname{Fix}_{\MapS}(M)\leq\operatorname{Fix}(M_0).\] Thus \(\rho_M\) is injective on each \(\langle h\rangle\cap\operatorname{Stab}_{\MapS}(M)\). For \(h\in F\), write
\[
 A_h=\rho_M\bigl(\langle h\rangle\cap\operatorname{Stab}_{\MapS}(M)\bigr).
\]
These are cyclic subgroups of \(\operatorname{Map}(\widehat M)\).
We claim that a puncture-fixing pseudo-Anosov \(z\) can be chosen so that, for all sufficiently large \(s\),
\begin{equation}
 \label{eq:TL1-avoidance}
 A_h\cap E(z^s\bar fz^{-s})=\{1\}
 \qquad(h\in F).
\end{equation}

The finite case of this claim is illustrated in Figure~\ref{fig:TL1-separation}.
\begin{figure}[!ht]
\centering
\begin{tikzpicture}[x=1cm,y=1cm,font=\small,>=stealth,
  minus/.style={circle,draw=blue!45!black,fill=blue!45!black,inner sep=1.7pt},
  plus/.style={circle,draw=blue!45!black,fill=white,inner sep=1.7pt}]
  \node[anchor=west] at (0,4.95) {(a) Avoiding a finite $A_h$: convergence in $\partial\mathcal C(\widehat M)$};
  \node[minus,label=above:{$\bar f^-$}] at (1.0,4.10) {};
  \node[plus,label=below:{$\bar f^+$}] at (1.0,2.90) {};
  \node[minus,label=above:{$z^s\bar f^-$}] at (4.1,3.95) {};
  \node[plus,label=below:{$z^s\bar f^+$}] at (4.1,3.05) {};
  \node[minus] at (7.0,3.80) {};
  \node[plus] at (7.0,3.20) {};
  \node[minus] at (8.8,3.65) {};
  \node[plus] at (8.8,3.35) {};
  \draw[->,densely dashed,black!55] (1.4,4.12)--(3.55,4.01);
  \draw[->,densely dashed,black!55] (1.4,2.88)--(3.55,2.99);
  \draw[->,densely dashed,black!55] (4.55,3.92)--(6.55,3.83);
  \draw[->,densely dashed,black!55] (4.55,3.08)--(6.55,3.17);
  \draw[->,densely dashed,black!55] (7.4,3.77)--(8.35,3.68);
  \draw[->,densely dashed,black!55] (7.4,3.23)--(8.35,3.32);
  \draw[->,densely dashed,black!55] (9.20,3.64)--(10.1,3.53);
  \draw[->,densely dashed,black!55] (9.20,3.36)--(10.1,3.47);
  \fill (10.45,3.50) circle (.045);
  \node[anchor=west] at (10.55,3.50) {$z^+$};
  \node[font=\footnotesize] at (8.1,4.36) {$s\longrightarrow+\infty$};
  \node[align=center,font=\footnotesize] at (5.8,2.03)
    {One $1\ne a_h'\in A_h$ preserving infinitely many pairs would give\\[3pt]
     $a_h'z^+=z^+\quad\Longrightarrow\quad a_h'\in A_h\cap E(z)=\{1\}$.};
\end{tikzpicture}
\caption{
(a) For finite $A_h$, there are only finitely many positive exponents $s$ that can be exceptional. The arrows denote convergence, not paths in the boundary.}

\label{fig:TL1-separation}
\end{figure}

Apply \cite[Proposition~1.3]{AD18} to the nontrivial finite groups among the \(A_h\), for the cobounded, nonelementary acylindrical curve graph action
\cite{B08}. Its hypothesis on quasi-fixed sets holds because
\(\operatorname{Map}(\widehat M)\) has no nontrivial finite normal subgroup, as shown in the proof of \cite[Theorem~2.2]{AD18}. The proposition supplies infinitely many independent pseudo-Anosovs whose elementary subgroups meet every finite \(A_h\) trivially. Choose one, denoted \(z\), independent from \(\bar f\), and replace it by a positive power fixing every puncture. Then \(E(z)\ne E(\bar f)\) and
\(A_h\cap E(z)=\{1\}\) for every finite \(A_h\). If there are no nontrivial finite \(A_h\), simply choose an independent pseudo-Anosov and take a puncture-fixing power.

If \(A_h\) is infinite with reducible generator, its nontrivial elements cannot lie in the elementary subgroup of a pseudo Anosov. If a generator \(a_h\) is pseudo-Anosov, a nontrivial intersection in \eqref{eq:TL1-avoidance} would give 
\(E(a_h)=E(z^s\bar fz^{-s})\). This can occur for at most one integer \(s\): two such equalities would give a nonzero power of \(z\) in \(E(\bar f)\), contrary to \(E(z)\ne E(\bar f)\). If \(A_h\) is finite and the intersection is nontrivial for infinitely many positive \(s\), a fixed \(1\ne a_h'\in A_h\) preserves \(z^s\{\bar f^-,\bar f^+\}\) along a subsequence. Both endpoints tend to \(z^+\), by north--south dynamics and disjointness of the endpoint pairs of \(z\) and \(\bar f\).
Consequently \(a_h'z^+=z^+\), which implies \(a_h'\in E(z)\) for the acylindrical curve-graph action, a contradiction. This proves the claim simultaneously for the finite family \(A_h\).

Lift \(z\) to \(\widetilde z\in\operatorname{Map}(M,\partial M)\) using the capping homomorphism \cite[Proposition~3.19]{FM11}, and set \(g=\widetilde z^sf\widetilde z^{-s}\) for an \(s\) satisfying \eqref{eq:TL1-avoidance}. The element \(g\) is WWPD loxodromic on \(\mathbb X\). If a nonidentity power \(h^m\) sent an endpoint of \(g\) to an endpoint of \(g\), Proposition~\ref{prop:TL1-support} would give \(h^m\in\operatorname{Stab}_{\MapS}(M)\) and \(\rho_M(h^m)\in A_h\cap E(\bar g)=\{1\}\). Injectivity on \(\langle h\rangle\cap\operatorname{Stab}_{\MapS}(M)\)
would then give \(h^m=1\), a contradiction.
\end{proof}

The next proposition turns endpoint separation into the required free-product.

\begin{proposition}
\label{prop:wwpd-combination}
Let \(G\) act on the path-metric realization of a connected hyperbolic graph \(X\), let \(g\) be WWPD loxodromic, and let \(1\ne h\in G\). Suppose that
\[h^m\{g^-,g^+\}\cap\{g^-,g^+\}=\varnothing \quad(m\in\mathbb Z,\ h^m\ne1).\]
For every sufficiently large integer \(Q\), the natural map gives
\[\langle h,g^Q\rangle=\langle h\rangle*\langle g^Q\rangle,\]
and every element not conjugate into \(\langle h\rangle\) is loxodromic on \(X\). In particular, every elliptic subgroup of \(\langle h,g^Q\rangle\) is conjugate into \(\langle h\rangle\).
\end{proposition}

\begin{proof}
Choose a \(g\)-invariant periodic quasi-axis \(\mathcal A\), obtained by translating a finite edge path, and fix \(x\in\mathcal A\). For a subset \(Z\subset X\), let
\[\pi_Z(y)=\{p\in Z:d(y,p)\leq d(y,Z)+1\},\quad \pi_Z(B)=\bigcup_{y\in B}\pi_Z(y).\]
These are metric projections in \(X\). We use the following standard consequence of thin triangles: for quasiconvex \(Z\) and \(p\in Z\), if \(d(y,p)\leq d(y,Z)+\kappa\) with \(\kappa\geq1\), then
\begin{equation}
 \label{eq:TL1-projection}
 d(z,p)+d(p,y)\leq d(z,y)+C_\kappa\quad(z\in Z),
\end{equation}
where \(C_\kappa\) depends only on hyperbolicity, quasiconvexity, and \(\kappa\). Indeed, consider a point \(w\in[z,p]\) sufficiently far from \(p\). The point \(w\) cannot be close to \([p,y]\). Otherwise, quasiconvexity of \(Z\) would yield a point \(q\in Z\) satisfying
\[
d(y,q)<d(y,p)-\kappa.
\]This would contradict \[d(y,p)\le d(y,Z)+\kappa\le d(y,q)+\kappa.\]
Thinness of the triangle therefore forces \(w\) to be close to \([z,y]\). This proves the inequality. The estimate also uniformly bounds the diameter of coarse projections of one point. For \(p\in\pi_{\mathcal A}(y)\), equation \eqref{eq:TL1-projection} gives the Gromov product estimate
\[
(y\mid p)_x\geq d(x,p)-C_1/2.
\] Thus, if such projections converge along \(\mathcal A\) to an endpoint, the original points converge to the same endpoint.

By the bounded-projection characterization WWPD(3) in \cite[Proposition and Definition~2.3]{HM}, the sets \(\pi_{\mathcal A}(h^m\mathcal A)\), over all nonidentity powers \(h^m\), have uniformly bounded diameter. Every point of \(X\) has a closest point on \(\mathcal A\). To see this, let \(I\) be a finite fundamental path such that
\[
\mathcal A=\bigcup_{n\in\mathbb Z}g^nI.
\]For every \(y\in X\), since \(g\) is loxodromic, we have
\[
d(y,g^nI)\longrightarrow\infty
\quad\text{as }|n|\longrightarrow\infty.
\]
Hence the minimum of \(d(y,g^nI)\) is attained for some \(n\). The path \(g^nI\) is compact, so there exists \(p\in g^nI\subset\mathcal A\) such that
\(
d(y,p)=d(y,\mathcal A).
\)
Equation~\eqref{eq:TL1-projection} allows us to use all coarse projections in place of a closest-point selection. In fact, these projection sets lie in a common bounded neighborhood of \(x\). Otherwise, their uniform diameter bound would give \(p_i\in\pi_{\mathcal A}(h^{m_i}x)\) tending to some \(\epsilon\in\{g^-,g^+\}\), hence \(h^{m_i}x\to\epsilon\). Since
\[
 d(h^{m_i}x,h^{m_i+1}x)=d(x,hx),
\]
continuity of the boundary action would imply \(h\epsilon=\epsilon\), contrary to the hypothesis. For finite-order \(h\), boundedness also follows directly from the finite number of powers. Applying the same argument to inverse powers and translating, choose one \(R\) such that, for every \(h^m\ne1\),
\[
 \pi_{\mathcal A}(h^m\mathcal A)\subset B(x,R),\quad
 \pi_{h^m\mathcal A}(\mathcal A)\subset B(h^mx,R).
\]
A pair of points realizing the distance between the two axes to within \(1\) belongs to these two projection sets. Therefore, with \(\kappa=2R+1\),
\begin{equation}
 \label{eq:TL1-joining-segment}
 d(x,h^mx)\leq d(\mathcal A,h^m\mathcal A)+\kappa.
\end{equation}

Consequently, there is a constant \(C\), independent of \(m\), such that the path consisting of a geodesic with endpoints on \(\mathcal A\), the joining segment \([x,h^mx]\), and a geodesic with endpoints on \(h^m\mathcal A\) is a \((1,C)\)-quasi-geodesic whenever the adjacent geodesics end at \(x\) and start at \(h^mx\). For completeness, if \(u\in\mathcal A\), \(v\in h^m\mathcal A\), and \(p\in\pi_{\mathcal A}(v)\), then \(d(p,x)\leq R\) and \eqref{eq:TL1-joining-segment} makes \(h^mx\) a \((R+\kappa)\)-coarse projection of \(p\) to \(h^m\mathcal A\). Two applications of \eqref{eq:TL1-projection} yield
\[
\begin{aligned}
 d(u,x)+d(x,h^mx)+d(h^mx,v)
 &\leq d(u,p)+d(p,h^mx)+d(h^mx,v)+2R\\
 &\leq d(u,v)+C.
\end{aligned}
\]

Let \(D\) be a common quasiconvexity constant for
\(\mathcal A\) and its translates. If
\(
u'\in[u,x]\) and \(
v'\in[h^mx,v],
\)
choose \(a\in\mathcal A\) and \(b\in h^m\mathcal A\) such that
\(
d(u',a)\leq D\) and \(d(v',b)\leq D.\)
Applying the preceding estimate to \(a\) and \(b\) gives
\[
d(u',x)+d(x,h^mx)+d(h^mx,v')
\leq d(u',v')+C+4D.
\]
After replacing \(C\) by \(C+4D\), this estimate holds for all \(u'\in[u,x]\) and \(v'\in[h^mx,v]\). 

If an endpoint lies inside \([x,h^mx]\), equation~\eqref{eq:TL1-joining-segment} shows that the appropriate end of that segment is its \(\kappa\)-coarse projection to the adjacent axis, and \eqref{eq:TL1-projection} gives the same conclusion. Thus the estimate holds for every subpath containing at most one joining segment, including a truncated one.

Let \(L\) be the local scale in the local-to-global theorem for \((1,C)\)-quasi-geodesics \cite[Proposition~7.2.E]{Gro87}. Choose \(Q\) with \(Q\tau_X(g)>L\), where \(\tau_X(g)>0\) is the stable translation length. Consider the homomorphism
\[\phi:\langle h\rangle*\langle T\rangle\longrightarrow G,\quad \phi(h)=h, \phi(T)=g^Q.\]
A cyclically reduced mixed word, after a cyclic permutation, has the form
\[ v=T^{n_1}h^{m_1}\cdots T^{n_k}h^{m_k},\quad n_j\ne0,h^{m_j}\ne1.\]
Join successive orbit points by translates of \([x,g^{Qn_j}x]\) and \([x,h^{m_j}x]\), following the syllables of \(v\), and extend the path periodically under \(\phi(v)\). Between consecutive joining segments it travels along a geodesic with endpoints on one translate of \(\mathcal A\), of length at least \(Q|n_j|\tau_X(g)>L\). Every subpath of length at most \(L\) therefore satisfies the preceding \((1,C)\) estimate, including across the end of each period. The local-to-global theorem makes the entire periodic path a quasi-geodesic. Since \(\phi(v)\) shifts it by a positive period, \(\phi(v)\) is loxodromic.

See Figure~\ref{fig:TL1-combination} for the local estimate and the resulting periodic path.
\begin{figure}[htbp]
\centering
\begin{tikzpicture}[x=1cm,y=1cm,font=\small,>=stealth,
  axis/.style={densely dashed,black!55,line width=.5pt},
  travel/.style={blue!45!black,line width=1.25pt},
  connector/.style={black,line width=.8pt}]
  \node[anchor=west] at (0,5.6) {(a) A uniform estimate at one joining segment};
  \draw[axis] (.25,4.65)--(1.1,4.70)..controls (1.7,4.95) and (2.70,4.30)..(3.50,4.55)--(4.45,4.60);
  \draw[axis] (7.1,3.50)--(8.10,3.48)..controls (9.0,3.22) and (10.1,3.88)..(10.95,3.62)--(11.8,3.50);
  \node at (2.4,5.10) {$\mathcal A$};
  \node at (9.4,4.00) {$h^m\mathcal A$};
  \fill[black!7] (3.50,4.55) ellipse (.47 and .30);
  \fill[black!7] (8.10,3.48) ellipse (.47 and .30);
  \draw[travel,->] (1.1,4.70)--(3.50,4.55);
  \draw[connector,->] (3.50,4.55)--(8.10,3.48);
  \draw[travel,->] (8.10,3.48)--(10.95,3.62);
  \draw[black!35,densely dotted] (1.1,4.70)--(10.95,3.62);
  \node[font=\footnotesize,fill=white,inner sep=1pt] at (6.1,4.38) {$d(u,v)$};
  \foreach \xx/\yy in {1.1/4.70,3.50/4.55,8.10/3.48,10.95/3.62}
    \fill (\xx,\yy) circle (.038);
  \node[anchor=south east] at (1.03,4.80) {$u$};
  \node[anchor=south] at (3.5,4.82) {$x$};
  \node[anchor=north] at (8.1,3.20) {$h^mx$};
  \node[anchor=north west] at (11.02,3.52) {$v$};
  \node[font=\footnotesize] at (3.5,4.00) {$B(x,R)$};
  \node[font=\footnotesize,anchor=west] at (8.60,3.03) {$B(h^mx,R)$};
  \node at (6,2.43)
    {$d(u,x)+d(x,h^mx)+d(h^mx,v)\leq d(u,v)+C$};
  \node[anchor=west] at (0,1.84) {(b) Long axis segments separate successive joining segments};
  \draw[travel,densely dashed] (0,.7)--(.5,.7);
  \draw[travel,->] (.5,.7)--(3.10,.7);
  \draw[connector,->] (3.10,.7)--(4.10,-.05);
  \draw[travel,->] (4.10,-.05)--(7.45,-.05);
  \draw[connector,->] (7.45,-.05)--(8.55,.7);
  \draw[travel,->] (8.55,.7)--(11.5,.7);
  \draw[travel,densely dashed] (11.5,.7)--(12,.7);
  \draw[black!60] (4.10,-.40)--(4.10,-.54)--(7.45,-.54)--(7.45,-.40);
  \node[font=\footnotesize,anchor=north] at (5.775,-.65)
    {length $\geq Q|n_j|\tau_X(g)>L$};
  \draw[black!50,densely dashed,rounded corners=7pt]
    (2.62,-.33) rectangle (4.57,1.05);
  \node[font=\footnotesize,align=center] at (3.595,1.43)
    {a subpath of arclength $\leq L$};
  \node[font=\footnotesize] at (9.45,-.56) {repeat under $\phi(v)$};
  \draw[travel] (.2,-1.63)--(1,-1.63);
  \node[anchor=west,font=\footnotesize] at (1.15,-1.63)
    {translate of $[x,g^{Qn_j}x]$};
  \draw[connector] (6.7,-1.63)--(7.5,-1.63);
  \node[anchor=west,font=\footnotesize] at (7.65,-1.63)
    {translate of $[x,h^{m_j}x]$};
\end{tikzpicture}
\caption{From the local estimate to a periodic quasi-geodesic.
(a) For points $u\in\mathcal A$ and $v\in h^m\mathcal A$, bounded mutual projections give the displayed excess-length estimate. (b) For the mixed word $v$, joining segments are separated by geodesic segments of length greater than $L$. Every subpath of arclength at most $L$ meets at most one joining segment ans local-to-global yields a periodic quasi-geodesic translated by $\phi(v)$. The drawings are schematic: geodesic segments have endpoints on the quasi-axes, and joining segments need not be short.}
\label{fig:TL1-combination}
\end{figure}

Every nontrivial word in the abstract free product is conjugate either into one of its factors or to such a mixed word. The map \(\phi\) is injective on both factors, and the preceding argument handles the mixed words. Hence \(\phi\) is injective, and every element not conjugate into \(\langle h\rangle\) is loxodromic. Finally, let \(B\leq\langle h,g^Q\rangle\) be elliptic on \(X\), and consider its action on the Bass--Serre tree of this free product. By the preceding conclusion, each nonidentity element of \(B\) is conjugate into \(\langle h\rangle\), and hence fixes a unique vertex of \(\langle h\rangle\)-type, since edge stabilizers are trivial. Two such elements fixing different vertices would have a product hyperbolic on the tree, hence not conjugate into \(\langle h\rangle\), and therefore loxodromic on \(X\), a contradiction. Thus the subgroup fixes a common vertex and is contained in a conjugate of \(\langle h\rangle\).
\end{proof}

\begin{proof}[Proof of Theorem \ref{TL1}.]
Let \(F\subset\MapS\setminus\{1\}\) be finite. By Proposition~\ref{prop:TL1-separation}, there is one WWPD loxodromic \(g\) on the fixed space \(\mathbb X\) satisfying \eqref{eq:TL1-separation}. Apply Proposition~\ref{prop:wwpd-combination} to every \(h\in F\), and choose a common sufficiently large exponent \(Q\). Then \(\gamma=g^Q\) is loxodromic on \(\mathbb X\), and for all \(h\in F\),
\[
 \langle h,\gamma\rangle=\langle h\rangle*\langle\gamma\rangle,
\]
with every elliptic subgroup conjugate into \(\langle h\rangle\). This is exactly Definition~\ref{naive1}. The choices of \(M\), \(g\), and \(Q\) depend on \(F\), while \(K\), the projection data, \(\mathbb X\), and its action remain fixed.
\end{proof}

Notice that \(P^X_{\mathrm{naive}}\) immediately implies the usual property \(P_{\mathrm{naive}}\), since one simply forgets the action on \(X\) and keeps the free product conclusion. The point of Theorem~\ref{TL1} is therefore not to reprove \(P_{\text{naive}}\) from scratch, but to strengthen the argument by showing that the same conclusion can be obtained with an action of \(\MapS\) on a fixed hyperbolic space \(\mathbb X\).

Combining Proposition \ref{P1} and Theorem \ref{TL1}, the group $\operatorname{Map}(S)$ has the following result.

\begin{corollary}\label{C1}
Let \(S\) be a connected orientable infinite type surface without boundary containing a connected nondisplaceable finite type subsurface. Then there exists a fixed hyperbolic \(\mathbb X\) such that \(\MapS\) has property \(\pa^{\mathbb X}\).
\end{corollary}

By Lemma \ref{a} and Lemma \ref{f}, the mapping class group $\MapS$ acts faithfully on $\mathbb{X}$ and induces a faithful action representation $\pi: \MapS \rightarrow \mathcal{U}(\ell^2(\mathbb{X}))$. Therefore, the action-based $C^{*}$-algebra $C^{*}_{\mathbb{X}}\MapS$ of $\MapS$ is the norm closure of $\pi(\mathbb{C}\MapS)$ in $\mathcal{B}(\ell^2(\mathbb{X}))$. 

Although \(\operatorname{Map}(S)\) has property \(\pn^{\mathbb X}\), and hence property \(\pa^{\mathbb X}\), the following proposition shows that the action representation associated with \(\operatorname{Map}(S)\curvearrowright\mathbb X\) is not coefficient-separating. Moreover, the corresponding action-based \(C^*\)-algebra \(C_{\mathbb X}^*(\operatorname{Map}(S))\) is not simple.

\begin{proposition}\label{simple555}
Let \(S\) be a connected orientable infinite type surface without boundary containing a connected nondisplaceable finite type subsurface. Then there exists a fixed hyperbolic \(\mathbb X\) such that the action representation of \(\operatorname{Map}(S)\) on $\ell^2(\mathbb X)$ is not coefficient-separating, and $C_{\mathbb X}^*\operatorname{Map}(S)$ is not simple.
\end{proposition}

\begin{proof}
Choose a vertex $x\in\mathcal C_S(K)\subset\mathbb X$, and let $\delta_x\in\ell^2(\mathbb X)$ be the Dirac vector. Since $S$ is of infinite type and $K$ is of finite type, the complement of $K$ contains essential curves $\alpha,\beta$ with $i_S(\alpha,\beta)\ge2$. By \cite[Theorem~3.14]{FM11}, their Dehn twists $a=T_\alpha$ and $b=T_\beta$ freely generate
\[
H=\langle a,b\rangle\cong F_2.
\]
Both twists have representatives supported outside $K$, so they fix $x$ and hence satisfy $\pi(a)\delta_x=\pi(b)\delta_x=\delta_x$.

Write $B=C^*(\pi(a),\pi(b))\subset C_{\mathbb X}^*\operatorname{Map}(S)$ and define
\[\chi:B\longrightarrow\mathbb C, \qquad \chi(T)=\langle T\delta_x,\delta_x\rangle.\]
Since the generators of $B$ and their adjoints fix $\delta_x$, the line $\mathbb C\delta_x$ is reducing for $B$, and
\[T\delta_x=\chi(T)\delta_x \qquad(T\in B).\]
Consequently, $\chi$ is a unital $*$-homomorphism, that is, a character of $B$.

By Corollary~\ref{C1} and Proposition~\ref{prop:regular-quotient}, there is a canonical surjective $*$-homomorphism
\[q:C_{\mathbb X}^*\operatorname{Map}(S)\longrightarrow C_r^*(\operatorname{Map}(S)_d), \qquad q(\pi(s))=\lambda_{\operatorname{Map}(S)}(s).\]
The restriction of $\lambda_{\operatorname{Map}(S)}$ to $H$ is a direct sum of copies of $\lambda_H$, obtained from the right-coset decomposition $\operatorname{Map}(S)=\bigsqcup_{Ht\in H\backslash \operatorname{Map}(S)}Ht$. Thus
\[
q(B)=C^*(\lambda_{\operatorname{Map}(S)}(a),\lambda_{\operatorname{Map}(S)}(b))
\cong C_r^*(H)\cong C_r^*(F_2).
\]

Suppose that $q$ is injective. Then $q|_B$ identifies $B$ with $C_r^*(F_2)$, which is simple by Powers' theorem \cite{Pow75}. Since $\chi(1)=1$ and the ideal $\ker\chi$ is proper, the simplicity implies $\ker\chi=\{0\}$. Hence $\chi$ is injective, forcing $B$ to be commutative, contrary to the noncommutativity of $C_r^*(F_2)$. Therefore $q$ is not injective. Since $q$ is unital, $\ker q$ is a nonzero proper closed two-sided ideal of $C_{\mathbb X}^*\operatorname{Map}(S)$, so $C_{\mathbb X}^*\operatorname{Map}(S)$ is not simple. Finally, Proposition~\ref{prop:regular-quotient} gives
\[\bigcap_{k\in \operatorname{Map}(S)}\ker E_k=\ker q\ne\{0\},\]
so the action representation is not coefficient-separating.
\end{proof}

Although the action-based \(C^*\)-algebra $C^*_{\mathbb X}\operatorname{Map}(S)$ is not simple, it nevertheless admits a unique normalized trace, as follows from Corollary~\ref{C1} and Proposition~\ref{Trace}.

\begin{corollary}
Let \(S\) be a connected orientable infinite type surface without boundary containing a connected nondisplaceable finite type subsurface. Then there exists a fixed hyperbolic \(\mathbb X\) such that the action-based $C^*$-algebra $C^*_{\mathbb X}\operatorname{Map}(S)$ admits a unique normalized trace.
\end{corollary}

\bibliographystyle{alpha}
\bibliography{sample}


\bigskip 
\footnotesize 
Tianyi Lou, 
\textsc{Universit\'e C\^ote d'Azur, CNRS, LJAD (UMR CNRS 7351), Parc Valrose, 06108 Nice Cedex 2, France}\par\nopagebreak 
Email address: \texttt{louty313@gmail.com}\par\nopagebreak Homepage: \url{https://sites.google.com/view/tianyilou/}

\end{document}